\documentclass[3p]{elsarticle}

\usepackage{amssymb,amsmath,amsthm}
\usepackage{color}
\usepackage{graphicx}
\usepackage{enumitem}
\usepackage{caption}
\usepackage{comment}
\usepackage[mathscr]{eucal}
\usepackage{multicol}
\usepackage{bm}
\usepackage{marginnote}
\usepackage[hang]{subfigure}
\usepackage{mathtools}
\usepackage{thmtools, thm-restate}
\usepackage{hyperref}

\hypersetup{
    colorlinks,
    linkcolor={blue!50!black},
    citecolor={green!50!black},
    urlcolor={red!80!black}
}

\providecommand{\E}[1]{{\ensuremath{\mathbb{E}}\mspace{-2mu}\left[#1\right]}}%
\providecommand{\work}[1]{{\ensuremath{\mathrm{Work}}\mspace{-2mu}\left[#1\right]}}%
\providecommand{\error}[1]{{\ensuremath{\mathrm{Error}}\mspace{-2mu}\left[#1\right]}}%

\providecommand{\cset}{\mathcal{C}}
\providecommand{\pset}{\mathbb{P}}
\providecommand{\rset}{\mathbb{R}}

\providecommand{\zset}{\mathbb{Z}}
\providecommand{\nset}{\mathbb{N}}

\providecommand{\Order}[1]{ {\ensuremath{ \mathcal O\left( #1 \right)}} }

\providecommand{\CWork}{\ensuremath{C_{\text{\textnormal{work}}}}}
\providecommand{\CError}{\ensuremath{C_{\text{\textnormal{error}}}}}
\renewcommand{\vec}[1]{{\ensuremath{{\boldsymbol #1}}}}
\providecommand{\valpha}{{\ensuremath{\vec{\alpha}}}}

\providecommand{\vbeta}{{\ensuremath{\vec{\beta}}}}
\providecommand{\veta}{{\ensuremath{\vec{\eta}}}}

\providecommand{\vXi}{{\ensuremath{\vec{\Xi}}}}

\newenvironment{changes}{\color[rgb]{0.7,0,0}}{}

\newcommand{\qoi}{F}
\newcommand{\descqoi}[2]{F_{#1,#2}}

\newcommand{\bb}{\vec b}
\newcommand{\cc}{\vec c}
\newcommand{\ee}{\vec e}

\newcommand{\jj}{\vec j}

\newcommand{\qq}{\vec q}

\newcommand{\vv}{\vec v}
\newcommand{\ww}{\vec w}
\newcommand{\xx}{\vec x}
\newcommand{\yy}{\vec y}

\newcommand{\zz}{\vec z}
\newcommand{\oone}{\bm{1}}
\newcommand{\di}[1]{~\text{d}#1}
\newcommand{\ellrad}{\xi}	%
\newcommand{\mcE}{\mathcal{E}}
\newcommand{\real}[1]{ \mathfrak{Re}\left( #1 \right)}
\newcommand{\im}[1]{\mathfrak{Im}\left( #1 \right)}

\newcommand{\ortp}{\Psi}	%
\newcommand{\oortp}{\bm{\Psi}}	%
\renewcommand{\div}{\textnormal{div}} 		%
\newcommand{\forcing}{\mathscr{F}}

\providecommand{\dRate}[1]{{r_{#1}}}

\definecolor{mygreen}{rgb}{0,0.6,0}
\usepackage{ulem}

\allowdisplaybreaks[1]

\theoremstyle{plain}

\declaretheorem{lemma}
\declaretheorem{problem}
\declaretheorem{assumption}
\declaretheorem{example}
\declaretheorem{remark}
\declaretheorem{definition}

\begin{document}
\begin{frontmatter}

\title{Multi-Index Stochastic Collocation for random PDEs}

\author[KAUST]{Abdul-Lateef~Haji-Ali\corref{cor1}}
\ead{abdullateef.hajiali@kaust.edu.sa}

\author[EPFL]{Fabio Nobile}
\ead{ fabio.nobile@epfl.ch}

\author[EPFL,UNIPV]{Lorenzo Tamellini}
\ead{lorenzo.tamellini@unipv.it}

\author[KAUST]{Ra\'ul Tempone}
\ead{raul.tempone@kaust.edu.sa}

\cortext[cor1]{Corresponding author} %

\address[KAUST]{CEMSE, King Abdullah University of Science
  and Technology, Thuwal 23955-6900, Saudi Arabia.  }
\address[EPFL]{CSQI - MATHICSE, \'Ecole Polytechnique F\'ed\'erale de
  Lausanne, Station 8, CH 1015, Lausanne, Switzerland}
\address[UNIPV]{Dipartimento di Matematica ``F. Casorati'',
  Universit\`a di Pavia, Via Ferrata 5, 27100 Pavia, Italy}

\begin{abstract}
In this work we introduce the Multi-Index Stochastic Collocation method (MISC)
for computing statistics of the solution of a PDE with random data.
MISC is a combination technique based on mixed differences of spatial
approximations and quadratures over the space of random data.
We propose an optimization procedure to select the most effective
mixed differences to include in the MISC estimator:
{such optimization is a crucial step and allows us to build a method that, provided with sufficient solution regularity,  is
potentially more effective than other multi-level collocation methods already available in literature}.
We then provide a complexity analysis that assumes decay rates of product type
for such mixed differences, {showing that in the optimal case the convergence rate of MISC is only dictated by the convergence of the deterministic solver applied to a one dimensional problem.}
We show the effectiveness of MISC with some computational tests,
comparing it with other related methods available in the literature,
such as the Multi-Index and Multilevel Monte Carlo, Multilevel Stochastic
Collocation, Quasi Optimal Stochastic Collocation and Sparse Composite
Collocation methods.
\end{abstract}


\end{frontmatter}

\section{Introduction}

Uncertainty Quantification (UQ) is an interdisciplinary, fast-growing
research area that focuses on devising mathematical techniques to
tackle problems in engineering and natural sciences in which only a
probabilistic description of the parameters of the governing equations
is available, due to measurement errors, intrinsic
non-measurability/non-predictability, or incomplete knowledge of the
system of interest. In this context, ``parameters'' is a term used in
broad sense to refer to constitutive laws, forcing terms, domain
shapes, boundary and initial conditions, etc.

UQ methods can be divided into deterministic
and randomized methods. While randomized techniques, which include the
Monte Carlo sampling method, are essentially based on random sampling
and ensemble averaging, deterministic methods proceed by building a
surrogate of the system's response function over the parameter space,
which is then processed to obtain the desired
information. Typical goals include computing statistical moments
(expected value, variance, higher moments, correlations) of some
quantity of interest of the system at hand, typically functionals of
the state variables (forward problem), or updating the
statistical description of the random parameters given some
observations of the system at hand (inverse problem).  In any case,
multiple resolutions of the governing equations are needed to
explore the dependence of the state variables on the random
parameters. The computational method used should therefore be
carefully designed to minimize the computational effort.

In this work, we focus on the case of PDEs with random data, for which
both deterministic and randomized approaches have been extensively explored in recent years.
As for the deterministic methods, we mention here the methods based on polynomial
expansions computed either by global Galerkin-type projections
\cite{ghanem.spanos:book,lemaitre:book,matthies.keese:galerkin,todor.schwab:convergence,xiu.karniadakis:wiener}
or collocation strategies based on sparse grids
(see e.g. \cite{babuska.nobile.eal:stochastic2,b.griebel:acta,nobile.tempone.eal:aniso,xiu.hesthaven:high}),
low-rank techniques \cite{khoromskij.schwab:tensor,khoromskij.oseledets:tensor,nouy:2008a,ballani.gras:tensor}
and reduced basis methods (see e.g. \cite{boyaval.patera:RBforPDE,chen.eal:RBvsSG}).
All these approaches have been found to be particularly effective when
applied to problems with a moderate number of random parameters
(low-dimensional probability space) and smooth response functions.
Although significant effort has been expended on increasing the
efficiency of such deterministic methods with respect to the number of
random parameters (see, e.g., \cite{cohen.devore.schwab:nterm2}, the seminal work on infinite
dimensional polynomial approximation of elliptic PDEs with random
coefficients), Monte Carlo-type approximations remain the primary
choice for problems with non-smooth response functions and/or those that depend on a
high number of random parameters, despite their slow convergence with respect to sample size.

A very promising methodology that builds on the classical Monte Carlo
method and enhances its performance is offered by the so-called
\emph{Multilevel Monte Carlo} (MLMC). It was first
proposed in \cite{Heinrich:MLMC} for applications in parametric
integration and extended to weak approximation of stochastic
differential equations in \cite{giles:MLMC}, which
also provided a full complexity analysis.
Let $\{h_\ell\}_{\ell=0}^L$ be a (scalar) sequence of spatial/temporal resolution
levels that can be used for the numerical discretization of the PDE at hand
and $\{\qoi_\ell\}_{\ell=0}^L$ be the corresponding
approximations of the quantity of interest,
and suppose that the final goal of the UQ analysis
is to compute the expected value of $\qoi$, $\E{\qoi}$.
While a classic Monte Carlo approach simply
approximates the expected value by using an ensemble average over a sample
of independent replicas of the random parameters, the MLMC method relies on the simple
observation that, by linearity of expectation,
\begin{equation}\label{eq:MLMC-telescopic}
  \E{\qoi} \approx \E{\qoi_L} =
\E{\qoi_0} + \sum_{\ell=1}^L\E{\qoi_\ell - \qoi_{\ell-1}},
\end{equation}
and computes by independent Monte Carlo samplers each expectation in the sum.
Indeed, if the discretization of the underlying
differential model is converging with respect to the
discretization level, $\ell$, the variance of $(\qoi_\ell-\qoi_{\ell-1})$ will be
smaller and smaller as $\ell$ increases, i.e., when the spatial/temporal
resolution increases. Dramatic computational saving can thus be obtained by approximating
the quantities $\E{\qoi_\ell-\qoi_{\ell-1}}$ with a smaller and smaller sample
size, since most of the variability of $\qoi$ will be captured with coarse simulations and
only a few resolutions over the finest discretization levels will be performed.
The MLMC estimator is therefore given by
\begin{equation}\label{eq:MLMC-estim}
  \E{\qoi} \approx \sum_{\ell=0}^L \frac{1}{M_\ell}\sum_{m=1}^{M_\ell} \left(\qoi_{\ell}(\omega_{m,\ell}) - \qoi_{\ell-1}(\omega_{m,\ell})\right),
  \quad \text{with } \qoi_{-1}(\cdot)=0,
\end{equation}
where $\omega_{m, \ell}$ are the i.i.d. replicas of the random parameters.
The application of MLMC methods to UQ problems involving PDEs with
random data has been investigated from the mathematical point of view
in a number of recent publications, see
e.g. \cite{bsz11,bls13,scheichl.charrier:MLMC,scheichl.giles:MLMC,mss12}.
Recent works \cite{teckentrup.etal:MLSC,van-wyk:MLSC,hps13,kss12} have
explored the possibility of replacing the Monte Carlo sampler on each
level by other quadrature formulas such as sparse grids or quasi-Monte Carlo quadrature,
obtaining the so-called Multilevel Stochastic Collocation
(MLSC) or Multilevel Quasi-Monte Carlo (MLQCM) methods.
See also \cite{tesei:MCCV} for a related
approach where the Multilevel Monte Carlo method is combined with a control
variate technique.

The starting point of this work is instead %
the so-called Multi-Index Monte Carlo method (MIMC), recently introduced in
\cite{abdullatif.etal:MultiIndexMC}, that differs from the Multilevel
Monte Carlo method in that the telescoping idea presented in equations
\eqref{eq:MLMC-telescopic}-\eqref{eq:MLMC-estim} is applied to
discretizations indexed by a multi-index rather
than a scalar index, thus allowing each discretization parameter to vary independently
of the others. Analogously to what done in
\cite{teckentrup.etal:MLSC,van-wyk:MLSC,hps13} in the context of
stochastic collocation, here we propose to replace
the Monte Carlo quadrature with a sparse grid quadrature at each
telescopic level, obtaining in our case the Multi-Index Stochastic
Collocation method (MISC). In other words, MISC can be seen
as a multi-index version of MLSC, or a stochastic collocation version of MIMC.
From a slightly different perspective, MISC is also closely related to the combination technique
developed for the solution of (deterministic) PDEs in
\cite{Bungartz.Griebel.Roschke.ea:pointwise.conv,b.griebel:acta,Griebel.schneider.zenger:combination,Hegland:combination,griebel.harbrecht:tensor};
in this work, the combination technique is used with respect to both
the deterministic and stochastic variables.

One key difference between the present work and
\cite{teckentrup.etal:MLSC,van-wyk:MLSC,hps13} is that the number of
problem solves to be performed at each discretization level is not
determined by balancing the spatial and stochastic components of the
error (based, e.g., on convergence error estimates), but rather suitably
extending the knapsack-problem approach that we employed in
\cite{nobile.eal:optimal-sparse-grids,back.nobile.eal:optimal,back.nobile.eal:lognormal}
to derive the so-called Quasi-Optimal Sparse Grids method (see also \cite{griebel.knapek:optimized}).
A somewhat analogous approach was proposed in \cite{bieri:sparse.tensor.coll}, where
the number of solves per discretization level is prescribed \emph{a-priori}
based on a standard sparsification procedure (we will give more details on
the comparison between these different methods later on).
{In this work, we provide a complexity analysis of MISC and illustrate its performance improvements,
comparing it to other methods by means of numerical examples.}

The remainder of this paper is organized as follows. In Section
\ref{s:problem-setting}, we introduce the problem to be solved and the
approximation schemes that will be used.  The Multi-Index Stochastic
Collocation method is introduced in Section \ref{s:method}, and our
main theorem detailing the complexity of MISC for a particular choice
of an index set is presented in Section \ref{s:complexity}.  Finally,
Section \ref{s:numerics} presents some numerical tests, while Section
\ref{s:conclusions} offers some conclusions and final remarks. The
Appendix contains the technical proof of the main theorem.  Throughout
the rest of this work we use the following notation:
\begin{itemize}
  \item $\nset$ denotes the set of integer numbers including zero;
  \item $\nset_+$ denotes the set of positive integer numbers, i.e. excluding zero;
  \item $\rset_+$ denotes the set of positive real numbers, $\rset_+ = \{ r \in \rset : r > 0\}$;
  \item $\oone$ denotes a vector whose components are always equal to one;
  \item $\vec e^\kappa_\ell$ denotes the $\ell$-th canonical vector in $\rset^\kappa$,
    i.e., $(\vec e^\kappa_\ell)_i = 1$ if $\ell=i$ and zero otherwise;
    however, for the sake of clarity, we often omit the superscript $\kappa$ when
    obvious from the context. For instance, if $\vv \in \rset^N$, we will write
    $\vv - \ee_1$ instead of $\vv - \ee_1^N$;
  \item given $\vv \in \rset^N$, $|\vv| = \sum_{n=1}^N v_n$,
    $\max(\vv) = \max_{n=1,\ldots N}v_n$ and $\min(\vv) = \min_{n=1,\ldots N}v_n$;
  \item given $\vv \in \rset^N$ and $f:\rset \to \rset$,
    $f(\vv)$ denotes the vector obtained by applying $f$ to each component of $\vv$,
    $f(\vv) = [f(v_1),f(v_2),\cdots,f(v_N)] \in \rset^N$;
  \item given $\vv, \ww \in \rset^N$, the inequality $\vv > \ww$ holds true if and only if $v_n > w_n$ $\forall n=1,\ldots,N$.
  \item given $\vv \in \rset^D$ and $\ww \in \rset^N$, $[\vv, \ww] = (v_1,\ldots,v_D,w_1,\ldots, w_N) \in \rset^{D+N}$.
\end{itemize}

\section{Problem setting}\label{s:problem-setting}

Let $\mathscr{B} \subset \rset^{d}$, $d=1,2,3$, be an open
hyper-rectangular domain (referred to hereafter as the ``physical domain'')
and let $\yy=(y_1,y_2,\ldots,y_N)$ be a $N$-dimensional random vector whose components
are mutually independent and uniformly distributed random variables with support
$\Gamma_n \subset \rset$ and probability density function $\rho_n(y_n) = \frac{1}{|\Gamma_n|}$.
Denoting $\Gamma = \Gamma_1 \times \Gamma_2 \ldots \times \Gamma_N$ (referred to hereafter as the
``stochastic domain'' or ``parameter space'') and by $\sigma_B(\Gamma)$ the Borel $\sigma$-algebra over $\Gamma$,
$\rho(\yy)d\yy = \prod_{n=1}^N \rho_n(y_n) dy_n$ is therefore
a probability measure on $\Gamma$, due to the independence of $y_n$,
and $(\Gamma,\sigma_B(\Gamma),\rho(\yy)d\yy)$ is a complete probability space.
Consider the following generic PDE, together with the assumption stated next:
\begin{problem}\label{pb:strong_form_yy}
Find $u:\mathscr{B} \times \Gamma \rightarrow \rset$ such that for
$\rho$-almost every $\yy \in \Gamma$
\[%
  \begin{cases}
    \mathcal{L}(u;\xx,\yy) = \forcing(\xx)  & \xx\in \mathscr{B},\\
    u(\xx,\yy)=0 %
 & \xx\in \partial \mathscr{B}.
  \end{cases}
\]
\end{problem}
\begin{assumption}[Well posedness]\label{assump:wellposed}
  Problem \ref{pb:strong_form_yy} is well posed in some Hilbert
  space $V$ for $\rho$-almost every $\yy \in \Gamma$.
\end{assumption}
The solution of Problem \ref{pb:strong_form_yy} can be
seen as an $N$-variate Hilbert-space valued function $u(\yy) : \Gamma \rightarrow V$.
The random variables, $y_n$, can represent scalar values whose exact value is unknown,
or they can stem from a spectral decomposition of a random field, like a
Karhunen-Lo\`eve or Fourier expansion, possibly truncated after a finite number of terms,
see, e.g., \cite{babuska.nobile.eal:stochastic2,back.nobile.eal:lognormal}.
It is also useful to introduce the Bochner space
$
 L^2_\rho(\Gamma; V) =
 \left\{ u: \Gamma \rightarrow V \mbox{ strongly measurable s.t. } \int_{\Gamma} \| u(\yy) \|_{V}^2 \rho(\yy)d\yy < \infty\right\}.
$
Finally, given some functional of the solution $u$, $\Theta: V \rightarrow \rset$,
we denote by $\qoi : \Gamma \rightarrow \rset$
the $N$-variate real-valued function assigning to each realization $\yy \in \Gamma$ the corresponding
value of $\Theta[u]$ (quantity of interest), i.e., $\qoi(\yy) = \Theta[u(\cdot,\yy)]$,
and we aim at estimating its expected value,
\[\E{\qoi} = \int_\Gamma \qoi(\yy) \rho(\yy) d\yy.\]
\begin{example}\label{example:ell}
As a motivating example, consider the following elliptic problem:
find $u:\mathscr{B} \times \Gamma \rightarrow \rset$
such that for $\rho$-almost every $\yy \in \Gamma$
\begin{equation}\label{eq:ell_PDE_strong_yy_example}
  \begin{cases}
    -\div( a(\xx,\yy)\nabla u(\xx,\yy) ) =\forcing(\xx)  & \xx\in \mathscr{B},\\
    u(\xx,\yy)=h(\xx) & \xx\in\partial \mathscr{B},
  \end{cases}
\end{equation}
holds, where $\div$ and $\nabla$ denote differentiation with respect to the physical
variables, $\xx$, only, and the function $a: \mathscr{B} \times \Gamma \rightarrow \rset$
is bounded away from $0$ and $\infty$, i.e., there exist two constants, $a_{min}, a_{max}$, such that
\begin{equation}\label{eq:bounding_a}
0< a_{min} \leq a(\xx,\yy) \leq a_{max} <\infty,
\quad \forall \xx \in \mathscr{B} \mbox{ and for } \rho\mbox{-almost every } \yy \in \Gamma.
\end{equation}
This boundedness condition guarantees that Assumption \ref{assump:wellposed} is satisfied,
i.e. the equation is well posed for $\rho$-almost every $\yy \in \Gamma$, thanks to a straightforward
application of the Lax-Milgram lemma;
moreover, the equation is well posed in $L^2_\rho(\Gamma; V)$, where $V$
is the classical Sobolev space $H^1_0(\mathscr{B})$, see, e.g., \cite{babuska.nobile.eal:stochastic2}.
This is the example we will focus on in Section \ref{s:numerics}, where we will test
numerically the performance of the Multi-Index Stochastic Collocation method that we will
detail in Section \ref{s:method}.
\end{example}

\begin{remark}
  The method that we present in the following sections can be also applied
  to more general problems than Problem \ref{pb:strong_form_yy}\, in
  which the forcing terms, boundary conditions and possibly domain
  shape are also modeled as uncertain; the extension to time-dependent
  problems with uncertain initial conditions is also
  straightforward. Other probability measures can also be considered;
  the very relevant case in which the random variables, $y_n$, are normally distributed
  is an example.
\end{remark}
\begin{remark}
As will be clearer in a moment, the methodology we propose uses tensorized
solvers for deterministic PDEs. Although for ease of exposition we have
assumed that the spatial domain, $\mathscr{B}$, is a hyper-rectangle, it is important
to remark that the methodology proposed in this work can also be applied to
non hyper-rectangular domains: this can be achieved by introducing a mapping from a reference
hyper-rectangle to the generic domain of interest (with techniques such as those
proposed in the context of Isogeometric Analysis \cite{hughes:IGA}
or Transfinite Interpolation \cite{gordon.hall:transfinite.interp})
or by a Domain Decomposition approach \cite{quarteroni1999domain}
if the domain can be obtained as a union of hyper-rectangles.
\end{remark}

\subsection{Approximation along the deterministic and stochastic dimensions} \label{s:interpolation}

In practice, we can only access the value of $\qoi$ via a numerical solver yielding a numerical
approximation of the solution $u$ of Problem \ref{pb:strong_form_yy},
which depends on a set of $D$ discretization parameters, such as the mesh-size,
the time-step, the tolerances of the numerical solvers, and others, which we denote by $h_i, i=1,\ldots,D$;
we remark that in general $D$, the number of parameters, might be
different from $d$, the number of spatial dimensions.
For each of those parameters, we introduce a sequence of discretization levels,
$h_{i,\alpha}, \alpha=1,2,\ldots$, and for each multi-index $\valpha \in \nset^D_+$, we denote by
$u^\valpha(\xx,\yy)$ the approximation of $u$ obtained from setting $h_i = h_{i,\alpha_i}$,
with the implicit assumption that $u^\valpha(\xx,\yy) \to u(\xx,\yy)$
as $\min_{1\leq i \leq D} \alpha_i \to \infty$ for $\rho$-almost every $\yy \in \Gamma$;
similarly, we also write $\qoi^\valpha(\yy) = \Theta[u^\valpha(\cdot,\yy)]$.
For instance, we could solve the problem stated in Example \ref{example:ell}
by a finite differences scheme with grid-sizes $h_{i,\alpha_i} = h_02^{-\alpha_i}$
in direction $i=1,\ldots,D$, for some $h_0 > 0$.

The discretization of $\qoi^\valpha$ over the random parameter space $\Gamma$
will consist of a suitable linear combination of tensor interpolants over $\Gamma$
based on Lagrangian polynomials. Observe that this approach is sound
only if $\qoi^\valpha$ is at least a continuous function over $\Gamma$
(the smoother $\qoi^\valpha$ is, the more effective the Lagrangian approximation will be);
for instance, for the problem stated in Example \ref{example:ell}, it can be shown under
moderate assumptions on $a(\xx,\yy)$ that $\qoi$ and $\qoi^\valpha$ are $\yy$-analytic, see, e.g.,
\cite{back.nobile.eal:optimal,cohen.devore.schwab:nterm2}; we will
return to this point in Section \ref{s:numerics}.

To derive a generic tensor Lagrangian interpolation of $\qoi^\valpha$, we first introduce the set
$\cset^0(\Gamma_n)$ of real-valued continuous functions over $\Gamma_n$,
and the subspace of polynomials of degree at most $q$ over $\Gamma_n$,
$\pset^q(\Gamma_n) \subset \cset^0(\Gamma_n)$.
Next, we consider a sequence of univariate Lagrangian interpolant
operators in each dimension $Y_n$, i.e.,
$\{\mathscr U_n^{m(\beta_n)}\}_{\beta_n \in \nset_+}$, where we refer
to the value $\beta_n$ as the ``interpolation level''.  Each interpolant
is built over a set of $m(\beta_n)$ collocation points,
$\mathscr{H}_n^{m(\beta_n)} = \{y_n^1, y_n^2 \ldots y_n^{m(\beta_n)}\} \subset \Gamma_n$,
where $m$ is a strictly increasing function, with $m(0)=0$ and $m(1)=1$,
that we call the ``level-to-nodes function''; thus, the interpolant yields
a polynomial approximation,
\[\mathscr U^{m(\beta_n)}_n : \cset^0(\Gamma_n) \rightarrow \pset^{m(\beta_n)-1}(\Gamma_n), \qquad
\mathscr U^{m(\beta_n)}_n[f](y_n) = \sum_{j=1}^{m(\beta_n)}
					\left( f(y_n^j) \prod_{k=1,k \neq j}^{m(\beta_n)} \frac{y_n - y_n^k}{y^j_n - y_n^k} \right),
\]
with the convention that $\mathscr U^{0}_n[f] = 0 \,\, \forall f \in \cset^0(\Gamma_n)$.

The $N$-variate Lagrangian interpolant can then be built by a tensorization
of univariate interpolants:
denote by $\cset^0(\Gamma)$ the space of real-valued $N$-variate continuous functions over $\Gamma$
and by $\pset^\qq(\Gamma) = \bigotimes_{n=1}^N \pset^{q_n}(\Gamma_n)$ the subspace of
polynomials of degree at most $q_n$ over $\Gamma_n$, with $\qq = (q_1,\ldots,q_N) \in \nset^N$,
and consider a multi-index $\vbeta \in \nset_+^N$ assigning the interpolation level in
each direction, $y_n$; the multivariate interpolant can then be written as
\begin{align*}
\mathscr{U}^{m(\vbeta)} : \cset^0(\Gamma) \rightarrow \pset^{m(\vbeta) -\oone}(\Gamma), \qquad
\mathscr{U}^{m(\vbeta)}[\qoi^\valpha] (\vec y)
&= \left( \mathscr{U}^{m(\beta_1)}_1 \otimes \cdots \otimes \mathscr{U}^{m(\beta_N)}_N \right) [\qoi^\valpha](\vec y).%
\end{align*}

The set of collocation points needed to build the
tensor interpolant $\mathscr{U}^{m(\vbeta)}[u] (\vec y)$
is the tensor grid $\mathscr{T}^{m(\vbeta)} = \times_{n=1}^N \mathscr{H}_n^{m(\beta_n)}$
with cardinality $ \#\mathscr{T}^{m(\vbeta)} = \prod_{n=1}^N m(\beta_n)$.
Observe that the Lagrangian interpolant immediately induces an $N$-variate quadrature formula,
$\mathscr{Q}^{m(\vbeta)} : \cset^0(\Gamma) \rightarrow \rset$,
\begin{align*}
& \mathscr{Q}^{m(\vbeta)}[\qoi^\valpha]
= \E{{\mathscr{U}}^{m(\vbeta)}[\qoi^\valpha](\vec y)}
= \sum_{j=1}^{\# \mathscr{T}^{m(\vbeta)} } \qoi^\valpha(\widehat{\yy}_j) \varpi_j,
\end{align*}
where $\widehat{\yy}_j \in \mathscr{T}^{m(\vbeta)}$ and the quadrature weights $\varpi_j$
are the expected values of the Lagrangian polynomials centered
in $\widehat{\yy}_j$, which can be computed exactly for most of the
common interpolation knots and probability measures of the random variables.

It is recommended that the collocation points $\mathscr{H}_n^{m(\beta_n)}$
to be used in each direction are chosen according to the underlying probability
measure, $\rho(y_n)dy_n$, to ensure good approximation properties
of the interpolant and quadrature operators, $\mathscr{U}^{m(\vbeta)}$ and $\mathscr{Q}^{m(\vbeta)}$.
Common choices are Gaussian quadrature points
like Gauss-Legendre for uniform measures or Gauss-Hermite for Gaussian measures,
cf. e.g., \cite{trefethen:book2013}, which are however
\emph{not nested}, i.e., $\mathscr{H}_n^{m(\beta_n)} \not\subset \mathscr{H}_n^{m(\beta_n+1)}$.
This means that they are not optimal for successive refinements %
of the interpolation/quadrature, and we will not consider them in this work.
Instead, we will work with \emph{nested} collocation points, and specifically with
Clenshaw-Curtis points \cite{nobile.eal:optimal-sparse-grids,trefethen:comparison},
that are a classical choice for the uniform measure that we are considering here;
other choices of nested points are available for uniform random variables,
e.g., the Leja points \cite{nobile.eal:optimal-sparse-grids,Chkifa:leja},
whose performance is somehow equivalent to that of Clenshaw-Curtis
for quadrature purposes, see \cite{nobile.etal:leja,narayan:Leja}.
Clenshaw-Curtis points are defined as
\begin{equation}\label{eq:CC_def}
  y_n^j=\cos\left(\frac{(j-1) \pi}{m(i_n)-1}\right), \quad 1\leq j \leq m(i_n),
\end{equation}
together with the following level-to-nodes relation, $m(i_n)$, that ensures their nestedness:
\begin{equation}\label{eq:m_CC}
m(0)=0,\,\, m(1)=1,\,\, m(i_n)=2^{i_n-1}+1.
\end{equation}
We conclude this section by introducing the following operator norm,
which acts as a ``Lebesgue constant''
from $\cset^0(\Gamma)$ to $L^2_\rho(\Gamma)$:
\begin{equation}\label{eq:leb_def}
\mathbb{M}^{m(\vbeta)} = \prod_{n=1}^N \mathbb{M}_n^{m(\beta_n)},
\quad \mbox{with } \quad
\mathbb{M}_n^{m(\beta_n)} = \sup_{\| f\|_{L^\infty(\Gamma_n)}=1} \| \mathscr{U}^{m(\beta_n)}_n f \|_{L^2_\rho(\Gamma_n)}.
\end{equation}
In particular, for the Clenshaw-Curtis points, it is possible to bound $\mathbb{M}_n^{m(\beta_n)}$ as:
\begin{equation}\label{eq:CC_leb_est}
  \mathbb{M}^{m(\vbeta)} \leq \mathbb{M}_{est}^{m(\vbeta)} = \prod_{n=1}^N \mathbb{M}_{n,est}^{m(\beta_n)},  \quad \quad
  \mathbb{M}_{n,est}^{q} =
  \begin{cases}
    1 & \mbox{ for } q=1 \\[2mm]
    \displaystyle \frac{2}{\pi}\log(q-1)+1 & \mbox{ for } q \geq 2.
  \end{cases}
\end{equation}
See \cite{nobile.eal:optimal-sparse-grids} and references therein.

\begin{remark}
  Nested collocation points have been studied also for other probability
  measures than uniform probability measures. In the very relevant case of a normal
  distribution, one possible choice is the Genz-Keister points
  \cite{genz.keister:kpnquad,back.nobile.eal:lognormal}; we mention
  also the recent work \cite{narayan:Leja} on generalized Leja points
  that can be used for arbitrary measures on unbounded domains.
\end{remark}

\section{Multi-Index Stochastic Collocation}\label{s:method}

It is easy to see that an accurate approximation of $\E{\qoi}$ by a
direct tensor technique as the one just introduced,
$\E{\qoi} \approx \mathscr{Q}^{m(\vbeta)} [\qoi^\valpha]$, might
require a prohibitively large computational effort even for moderate
values of $D$ and $N$ (what is referred to as the ``curse of
dimensionality'').  In this work, following the setting that was
presented in
\cite{nobile.eal:optimal-sparse-grids,abdullatif.etal:MultiIndexMC},
we propose the Multi-Index Stochastic Collocation as
an alternative. It can be seen as a generalization of the
telescoping sum presented in the introduction, see equations
\eqref{eq:MLMC-telescopic} and \eqref{eq:MLMC-estim}.  Denoting
$\mathscr{Q}^{m(\vbeta)} [\qoi^\valpha] = \descqoi{\valpha}{\vbeta}$,
the building blocks of such a telescoping sum are the first-order
difference operators for the deterministic and stochastic
discretization parameters, denoted respectively by
$\Delta_{i}^\textnormal{det}$ with $1\le i\le D$ and
$\Delta_{j}^\textnormal{stoc}$ with $1\le j\le N$:
\begin{align}\label{eq:misc_delta_operations}
& \Delta_{i}^\textnormal{det}[\descqoi{\valpha}{\vbeta}] =
\begin{cases}
  \descqoi{\valpha}{\vbeta} - \descqoi{\valpha-\ee_i}{\vbeta}, & \text{if }\alpha_i > 1,\\
  \descqoi{\valpha}{\vbeta} & \text{if } \alpha_i=1,
  \end{cases} \\% \label{eq:delta_det_def} \\
& \Delta_{j}^\textnormal{stoc}[\descqoi{\valpha}{\vbeta}]    =
\begin{cases}
  \descqoi{\valpha}{\vbeta} - \descqoi{\valpha}{\vbeta-\ee_j}, & \text{if }\beta_j > 1,\\
  \descqoi{\valpha}{\vbeta} & \text{if } \beta_j=1.
  \end{cases} %
\end{align}
We then define the first-order tensor difference operators,
\begin{align}
& \vec \Delta^\textnormal{det}[\descqoi{\valpha}{\vbeta}]
= \bigotimes_{i=1}^D \Delta_i^\textnormal{det}[\descqoi{\valpha}{\vbeta}]
= \Delta_1^\textnormal{det} \left[ \, \Delta_2^\textnormal{det} \left[ \, \cdots
  \Delta_D^\textnormal{det} \left[ \descqoi{\valpha}{\vbeta} \right] \, \right] \, \right] %
= \sum_{\jj \in \{0,1\}^D} (-1)^{|\jj|} \descqoi{\valpha-\jj}{\vbeta},
\label{eq:delta_det_def} \\
& \vec \Delta^\textnormal{stoc}[\descqoi{\valpha}{\vbeta}]
= \bigotimes_{j=1}^N \Delta_j^\textnormal{stoc}[\descqoi{\valpha}{\vbeta}]
= \sum_{\jj \in \{0,1\}^N} (-1)^{|\jj|} \descqoi{\valpha}{\vbeta-\jj} \, .
\label{eq:delta_stoc_def}
\end{align}
with the convention that $F_{\vec v, \vec w}=0$ whenever a component
of $\vec v$ or $\vec w$ is zero. Observe that computing
$\vec \Delta^\textnormal{det}[\descqoi{\valpha}{\vbeta}]$ actually
requires up to $2^D$ solver calls, and analogously applying
$\vec \Delta^\textnormal{stoc}[\descqoi{\valpha}{\vbeta}]$ requires
interpolating $\qoi^\valpha$ on up to $2^N$ tensor grids; for
instance, if $D=N=2$ and $\valpha,\vbeta > \oone$, we have
\begin{align*}
  &\vec \Delta^\textnormal{det} [\descqoi{\valpha}{\vbeta}]
  = \Delta^\textnormal{det}_2 \left[ \, \Delta^\textnormal{det}_1 \left[ \, \descqoi{\valpha}{\vbeta} \, \right] \, \right]
   = \Delta^\textnormal{det}_2 [ \descqoi{\valpha}{\vbeta} - \descqoi{\valpha-\ee_1}{\vbeta}]
   = \descqoi{\valpha}{\vbeta} - \descqoi{\valpha-\ee_1}{\vbeta}
   - \descqoi{\valpha-\ee_2}{\vbeta} + \descqoi{\valpha-\oone}{\vbeta}, \\[2pt]%
  &\vec \Delta^\textnormal{stoc} [\descqoi{\valpha}{\vbeta}]
  = \descqoi{\valpha}{\vbeta} - \descqoi{\valpha}{\vbeta-\ee_1}
   - \descqoi{\valpha}{\vbeta-\ee_2} + \descqoi{\valpha}{\vbeta-\oone}. \nonumber
\end{align*}
Finally, letting
$\vec \Delta[\descqoi{\valpha}{\vbeta}] =
\vec \Delta^\textnormal{stoc}[ \vec \Delta^\textnormal{det}[\descqoi{\valpha}{\vbeta}]]$,
we define the Multi-Index Stochastic Collocation (MISC) estimator of $\E{\qoi}$ as
\begin{equation}\label{eq:misc_estimator}
  \mathscr{M}_{\mathcal{I}}[\qoi]
  = \sum_{[\valpha, \vbeta] \in \mathcal I} \vec \Delta[\descqoi{\valpha}{\vbeta}]
  = \sum_{ [\valpha, \vbeta] \in \mathcal I}  c_{\valpha,\vbeta} \descqoi{\valpha}{\vbeta},
\end{equation}
where $\mathcal{I} \subset \nset_+^{D+N}$ and
$c_{\valpha,\vbeta} \in \zset$. %
Observe that many of the coefficients in \eqref{eq:misc_estimator}, $c_{\valpha,\vbeta}$,
may be zero: in particular, $c_{\valpha,\vbeta}$ is zero whenever
$[\valpha,\vbeta]+\vec{j} \in \mathcal{I}$
$\forall \vec{j}\in\{0,1\}^{N+D}$.  Similarly to the analogous sparse grid construction
\cite{nobile.eal:optimal-sparse-grids,wasi.wozniak:cost.bounds,b.griebel:acta},
we shall require that the multi-index set $\mathcal{I}$ be downward closed, i.e.,
\begin{align*}
  \forall \,[\valpha, \vbeta] \in \mathcal{I}, \quad
  \begin{cases}
    \valpha - \vec{e}_i \in \mathcal{I} \mbox{ for } 1 \leq i \leq D \mbox{ and } \alpha_i > 1,\\
    \vbeta - \vec{e}_j \in \mathcal{I} \mbox{ for } 1 \leq j \leq N \mbox{ and } \beta_j > 1.
  \end{cases}
\end{align*}

\begin{remark}
  In theory, a MISC approach could also be developed to approximate
  the entire solution $u(\xx,\yy)$ and not just the expectation of
  functionals, considering differences between consecutive interpolant
  operators, $\mathscr{U}^{m(\vbeta)}$, on the stochastic domain rather
  than differences of the quadrature operators,  $\mathscr{Q}^{m(\vbeta)}$,
  as a building block for the
  $\vec \Delta^\textnormal{stoc}$ operators, as well as considering the
  discretized solution $u^\valpha$ rather than just the quantity of
  interest, $\qoi^\valpha$, in the construction of the
  $\vec \Delta^\textnormal{det}$ operators.
\end{remark}

\subsection{A knapsack-like construction of the set $\mathcal{I}$}

The efficiency of the MISC method in equation \eqref{eq:misc_estimator}
will heavily depend on the specific
choice of the index set, $\mathcal{I}$; in the following, we will first
propose a general strategy to derive quasi-optimal sets and then
prove in Section \ref{s:complexity} a convergence result for such sets
under some reasonable assumptions.

To derive an efficient set, $\mathcal{I}$, we recast the problem of
its construction as an optimization problem, in the same spirit of
\cite{abdullatif.etal:MultiIndexMC,nobile.eal:optimal-sparse-grids,back.nobile.eal:optimal,b.griebel:acta,griebel.knapek:optimized}.
We begin by introducing the concepts of ``work contribution'', $\Delta W_{\valpha, \vbeta}$, and ``error contribution'',
$\Delta E_{\valpha, \vbeta}$, for each hierarchical surplus operator,
$\vec \Delta[\descqoi\valpha\vbeta]$. %
The work contribution measures the computational cost
(measured, e.g., as a function of the total number of degrees of freedom, or in terms of computational time)
required to add $\vec \Delta[\descqoi\valpha\vbeta]$ to $\mathscr{M}_\mathcal{I}[\qoi]$, i.e., to solve the
associated deterministic problems and to compute the corresponding interpolants over the parameter space,
cf. equations \eqref{eq:delta_det_def} and \eqref{eq:delta_stoc_def};
the error contribution measures instead how much the error $|\E{\qoi} - \mathscr{M}_\mathcal{I}[\qoi]| $
would decrease once the operator $\vec \Delta[\descqoi\valpha\vbeta]$
has been added to $\mathscr{M}_\mathcal{I}[\qoi]$. %
In formulas, we define
\[\Delta W_{\valpha, \vbeta}
= \work{\mathscr{M}_{\mathcal{I} \cup \{[\valpha,\vbeta]\}}[\qoi]} - \work{\mathscr{M}_{\mathcal{I}}[\qoi]}
= \work{\vec \Delta[\descqoi\valpha\vbeta]},\]
so that
\begin{equation}\label{eq:work_decomp}
  \work{\mathscr{M}_\mathcal{I}[\qoi]} = \sum_{[\valpha, \vbeta] \in \mathcal{I}} \Delta W_{\valpha, \vbeta},
\end{equation}
Observe that this work definition is sharp only if we think of building the MISC estimator
with an incremental approach,
i.e., we assume that adding the multi-index $(\valpha, \vbeta)$ to the index set $\mathcal I$ would not reduce the work that has to be done to evaluate the MISC estimator on the index set.
This implies that one cannot take advantage of the fact
that some of the coefficients in \eqref{eq:misc_estimator}, $c_{\valpha,\vbeta}$, that are non-zero
when considering the set $\mathcal{I}$ could become zero if the MISC estimator is instead
built considering the set $\mathcal{I} \cup \{[\valpha,\vbeta]\}$, hence it would be possible not to compute the
corresponding approximations $\descqoi{\valpha}{\vbeta}$.
This approach is discussed in Section~\ref{ss:impl}.

Similarly, we define
\[\Delta E_{\valpha, \vbeta}
 = \Big| {\mathscr{M}_\mathcal{I \cup \{[\valpha,\vbeta]\}}[\qoi]} - {\mathscr{M}_\mathcal{I}[\qoi]} \Big|
 = \left| {\vec \Delta[\descqoi\valpha\vbeta]} \right|.
\]
Thus, by construction, the error of the MISC estimator \eqref{eq:misc_estimator}
can be bounded as the sum of the error contributions not included in
the estimator $\mathscr{M}_\mathcal{I}[\qoi]$, %
\begin{align}\label{eq:error-decomp}
\error{\mathscr{M}_\mathcal{I}[\qoi]} =
|\E{\qoi} - \mathscr{M}_\mathcal{I}[\qoi]|
& = \left| {\sum_{[\valpha, \vbeta] \notin \mathcal I} \vec \Delta[\descqoi\valpha\vbeta]} \right| \nonumber \\[2mm]
 & \leq \sum_{[\valpha, \vbeta] \notin \mathcal I} \left| { \vec \Delta[\descqoi\valpha\vbeta] } \right|
  = \sum_{[\valpha, \vbeta] \notin \mathcal I} \Delta E_{\valpha, \vbeta}.
\end{align}
Consequently, a quasi-optimal set $\mathcal{I}$ can be computed by solving the following
``binary knapsack problem'' \cite{martello:knapsack}:
\begin{align}
  \mbox{maximize} &\sum_{[\valpha, \vbeta] \in \nset_+^{D+N}}\Delta E_{\valpha, \vbeta} x_{\valpha, \vbeta} \nonumber \\
  \mbox{such that }& \sum_{[\valpha, \vbeta] \in \nset_+^{D+N}}\Delta W_{\valpha, \vbeta} x_{\valpha, \vbeta} \leq W_{max}, \label{eq:knapsack} \\
  & \qquad x_{\valpha, \vbeta} \in \{0,1\}, \nonumber
\end{align}
and setting
$\mathcal{I}=\{ [\valpha, \vbeta] \in \nset_+^{D+N}: x_{\valpha,
  \vbeta} = 1\}$.
Observe that such a set is only ``quasi'' optimal since the error
decomposition \eqref{eq:error-decomp} is not an exact representation but
rather an upper bound.  The optimization problem above is well known
to be computationally intractable.  Still, an approximate greedy solution
(which coincides with the exact solution under certain hypotheses that will be
clearer in a moment) can be found if one instead allows the variables $x_{\valpha, \vbeta}$ to
assume fractional values, i.e., it is possible to include fractions of
multi-indices in $\mathcal{I}$. For this simplified problem, the resulting problem can be solved
analytically by the so-called Dantzig algorithm \cite{martello:knapsack}:
\begin{enumerate}
\item compute the ``profit'' of each hierarchical surplus, i.e., the quantity
  \[%
    P_{\valpha, \vbeta} = \frac{\Delta E_{\valpha, \vbeta}}{\Delta W_{\valpha, \vbeta}};
  \]
\item sort the hierarchical surpluses by decreasing profit;
\item add the hierarchical surpluses to $\mathscr{M}_\mathcal{I}[\qoi]$ according to such order
  until the constraint on the maximum work is fulfilled.
\end{enumerate}
Note that by construction $x_{\valpha, \vbeta}=1$ for all the multi-indices included
in the selection except for the last one, for which $x_{\valpha, \vbeta}<1$ might hold;
in other words, the last multi-index is the only one that might not be taken entirely.
However, if this is the case, we assume that we could slightly adjust the value $W_{max}$,
so that all $x_{\valpha, \vbeta}$ have integer values (see also \cite{b.griebel:acta});
observe that this integer
solution is also the solution of the original binary knapsack problem \eqref{eq:knapsack}
with the new value of $W_{max}$ in the work constraint.
Thus, if the quantities $\Delta E_{\valpha,\vbeta}$ and $\Delta W_{\valpha,\vbeta}$ were available,
the quasi-optimal index set for the MISC estimator could be computed as
\begin{equation}\label{eq:opt_set}
 \mathcal{I} = \mathcal{I}(\epsilon) =\left\{ [\valpha, \vbeta] \in \nset_+^{D+N}\::\:
   \frac{\Delta E_{\valpha,\vbeta}}{\Delta W_{\valpha,\vbeta}} \geq \epsilon \right\},
\end{equation}
for a suitable $\epsilon>0$.%

\begin{remark}\label{remark:webster_and_bieri}
  The MISC setting could in principle include the Multilevel Stochastic Collocation method
  proposed in \cite{teckentrup.etal:MLSC} as a special case, by simply considering
  a discretization of the spatial domain on regular meshes, and letting the diameter
  of each element (the mesh-size) be the only discretization parameter, i.e., $D=1$.

  However, the sparse grids to be used at each level
  are determined in \cite{teckentrup.etal:MLSC} by computing the minimal number of
  collocation points needed to balance the stochastic and spatial
  error.  This is done by relying on sparse grid error estimates; yet,
  since in general it is not possible to generate a sparse grid with a
  predefined number of points, some rounding strategy to the sparse
  grid with the nearest cardinality must be devised, which may affect
  the optimality of the multilevel strategy.  In the present work, we
  overcome this issue by relying instead on profit estimates to build
  a set of multi-indices that simultaneously prescribe the spatial
  discretization and the associated tensor grid in the stochastic
  variables.  Furthermore, only standard isotropic Smolyak sparse
  grids are considered in the actual numerical experiments in
  \cite{teckentrup.etal:MLSC} (although in principle anisotropic
  sparse grids could be used as well, provided that good convergence
  estimates for such sparse grids are available), while our
  implementation naturally uses anisotropic stochastic collocation
  methods at each spatial level.

  The MISC approach also includes as a special case the ``Sparse
  Composite Collocation Method'' developed in
  \cite{bieri:sparse.tensor.coll}, by considering again only one
  deterministic discretization parameter, i.e., $D=1$, and then
  setting
  \begin{equation}\label{eq:iso-MISC}
    \mathcal{I} = \left\{ [\alpha,\vbeta] \in \nset_+^{1+N} : \alpha + \sum_{n=1}^N \beta_n \leq w \right\},
  \end{equation}
  with $w \in \nset_+$.  In other words, the approach in \cite{bieri:sparse.tensor.coll}
  is based neither on profit nor on error balancing.
\end{remark}

\section{Complexity analysis of the MISC method}\label{s:complexity}

In this section, we assume suitable models for the error and work
contributions, $\Delta E_{\valpha,\vbeta}$ and
$\Delta W_{\valpha,\vbeta}$ (which are numerically verified in
Section \ref{s:numerics} for the problem in Example \ref{example:ell})
and then state our main convergence theorem for the MISC method built using a
particular index set, $\mathcal{I}^*$, which can be regarded as an
approximation of the quasi-optimal set introduced in the previous
section.

\begin{assumption}\label{assump:growth_of_dof}
  The discretization parameters, $h_i$, for the deterministic solver depend
  exponentially on the discretization level $\alpha_i$, and the number of collocation
  points over the parameter space grows exponentially with the level $\beta_i$:
  \[h_{i,\alpha_i}=h_0 2^{-\alpha_i} \qquad \text{and}\qquad  C_{m,low} 2^{\beta_i} \leq m(\beta_i) \leq C_{m,up} 2^{\beta_i}.\]
\end{assumption}
\begin{assumption}\label{assump:dW_dE_factor}
  The error and work contributions, $\Delta E_{\valpha,\vbeta}$ and $\Delta W_{\valpha,\vbeta}$,
  can be bounded as products of two terms,
  \[%
    \Delta E_{\valpha,\vbeta} \leq \Delta E_\valpha^\textnormal{det}  \Delta E_\vbeta^\textnormal{stoc}  \qquad \text{and}\qquad
    \Delta W_{\valpha,\vbeta} \leq \Delta W_\valpha^\textnormal{det}  \Delta
    W_\vbeta^\textnormal{stoc},  %
\]
  where $\Delta W_\valpha^\textnormal{det}$ and $\Delta E_\valpha^\textnormal{det}$ denote the cost and
  the error contributions due to the deterministic difference operator,
  $\Delta^\textnormal{det}[\descqoi{\valpha}{\vbeta}]$,
  and similarly $\Delta W_\vbeta^\textnormal{stoc}$ and $\Delta E_\vbeta^\textnormal{stoc}$ denote the
  cost and the error contribution
  due to the stochastic difference operator, $\Delta^\textnormal{stoc}[\descqoi{\valpha}{\vbeta}]$,
  cf. equations \eqref{eq:delta_det_def}-\eqref{eq:delta_stoc_def}.
\end{assumption}
\begin{assumption}\label{assump:dW_dE_model}
  The following bounds hold true for the factors appearing in Assumption \ref{assump:dW_dE_factor}:
  \begin{align}
    & \Delta W_\valpha^\textnormal{det} \leq C_{\mathrm{work}}^{\textnormal{det}} \prod_{i=1}^D (h_{i,\alpha_i})^{-\widetilde{\gamma}_i} , \label{eq:deltaW_alpha_est}\\
    & \Delta E_\valpha^\textnormal{det} \leq C_{\mathrm{error}}^{\textnormal{det}} \prod_{i=1}^D (h_{i,\alpha_i})^{\widetilde{\dRate{}_i} } , \label{eq:dE_alpha_est}\\
    & \Delta W_\vbeta^\textnormal{stoc} \leq \widetilde{C}_{\mathrm{work}}^{\textnormal{stoc}} \prod_{n=1}^N m(\beta_n)
    \leq C_{\mathrm{work}}^{\textnormal{stoc}}  \prod_{n=1}^N 2^{\beta_n},\label{eq:deltaW_beta_est}  \\
    & \Delta E_\vbeta^\textnormal{stoc} \leq C_{\mathrm{error}}^{\textnormal{stoc}}\, e^{- \sum_{i=1}^N \widetilde{g}_i m(\beta_i)}, \label{eq:deltaE_beta_est}%
  \end{align}
  for some rates $\widetilde{\gamma}_i, \widetilde{\dRate{}}_i, \widetilde{g}_i > 0$.
\end{assumption}

With these assumptions, we are now ready to state our main
theorem. The proof is technical and we therefore place it in the
appendix. The proof is based on summing the error contributions
outside a particular index set, $\mathcal I^*$, and the work
contributions inside the same index set. This can be seen as a
weighted cardinality argument in finite dimensions. See also
\cite{griebel:infdim, griebel:tensor} for similar arguments for
different choices of finite and infinite dimensional index sets.
\begin{restatable}[MISC computational complexity]{theorem}{misccomplexity}
\label{thm:misc_complexity}
Under Assumptions \ref{assump:growth_of_dof} to \ref{assump:dW_dE_model}, the bounds
for the factors appearing in Assumption \ref{assump:dW_dE_factor} can be equivalently rewritten as
  \begin{subequations}
  \begin{align}
    \label{eq:deltaW_ass}
    & \Delta W_{\valpha,\vbeta} \leq \CWork e^{\sum_{i=1}^D \gamma_i \alpha_i} e^{ \delta |\vbeta|}, \\
    \label{eq:deltaE_ass}
    & \Delta E_{\valpha,\vbeta} \leq \CError e^{ - \sum_{i=1}^D \dRate{i}\alpha_i} e^{- \sum_{j=1}^N g_j \exp(\delta {\beta_j})},
  \end{align}
  \end{subequations}
  with $\gamma_i = \widetilde{\gamma}_i\log 2$, $\dRate{i} =\widetilde{\dRate{}}_i\log 2$,
  $\delta = \log 2$ and $g_i = \widetilde{g}_i C_{m,low}$. Define the following set
  \begin{equation}\label{eq:opt_set_detailed_logform}
    \mathcal{I}^*(L) =
    \left\{ [\valpha, \vbeta] \in \nset_+^{D+N}\::\:
      \sum_{i=1}^D (\dRate{i}+\gamma_i) \alpha_i + \sum_{i=1}^N (\delta \beta_i +  g_i e^{\delta \beta_i} ) \leq L \right\} \mbox{ with } L \in \rset_+.
  \end{equation}
  Then there exists a constant $\mathscr C_{\text{W}}$ such that, for any $W_{\max}$ satisfying
  \begin{equation}
    \label{eq:maxW_ass}
    W_{\max} \geq \mathscr C_{\text{W}} \exp\left(\chi\right),
  \end{equation}
  and choosing $L$ as
  \begin{equation}\label{eq:opt-L}
  L = L(W_{\max}) = \frac{1}{\chi} \left(
    \log\left(\frac{W_{\max}}{\mathscr{C}_{\text{W}}}\right) - (\mathfrak{z}-1)
    \log\left(\frac{1}{\chi}\log\left(\frac{W_{\max}}{\mathscr{C}_{\text{W}}}\right)\right)
		\right),
  \end{equation}
  with $\vXi = \left( \frac{\gamma_1}{\gamma_1 + r_1},\ldots,\frac{\gamma_D}{\gamma_D + r_D} \right)$,
  $\chi = \max(\vXi)$,  $\zeta=\min_{i=1,\ldots,D} \frac{\dRate{i}}{\gamma_i}$ and
  $\mathfrak{z} = \#\{i=1,\ldots D \::\: \frac{\dRate{i}}{\gamma_i} =
  \zeta\}$, the MISC estimator $\mathscr M_{\mathcal I^*(L(W_{\max}))}$ satisfies
  \begin{subequations}
    \label{eq:misc_complexity}
    \begin{align}
      &\work{\mathscr M_{\mathcal I^*(L(W_{\max}))}} \leq W_{\max}, \label{eq:misc_complexity_work}
      \\[2mm]
      &\limsup_{W_{\max} \uparrow \infty} \frac
{\error{\mathscr M_{\mathcal I^*(L(W_{\max}))}}}{W_{\max}^{-\zeta}        \left(\log\left(W_{\max}\right)\right)^{\left(\zeta+1\right)\left(\mathfrak{z}-1\right)}}
        = \mathscr C_{\text{E}} < \infty. \label{eq:misc_complexity_error}
    \end{align}
  \end{subequations}
\end{restatable}

\begin{remark}\label{rem:this-is-the-set}
  The set $\mathcal{I}^*$ proposed in Theorem \ref{thm:misc_complexity} can be obtained
  by assuming that the bounds in equations \eqref{eq:deltaW_ass} and \eqref{eq:deltaE_ass}
  are actually equalities and by using the definition of the quasi-optimal set \eqref{eq:iso-MISC}:
  \begin{align*}
    \mathcal{I}^*
    & = \left\{ [\valpha,\vbeta] \in \nset_+^{D+N} : \frac{\Delta E_{\valpha,\vbeta}}{\Delta W_{\valpha,\vbeta}} \geq \epsilon \right \} \\
    & = \left\{ [\valpha,\vbeta] \in \nset_+^{D+N} :
      \frac{e^{ - \sum_{i=1}^D \dRate{i}\alpha_i} e^{- \sum_{j=1}^N g_j \exp(\delta {\beta_j})}}
      {e^{\sum_{i=1}^D \gamma_i \alpha_i} e^{ \delta |\vbeta|}} \geq \epsilon \right \} \\
    & = \left\{ [\valpha, \vbeta] \in \nset_+^{D+N}\::\:
      \sum_{i=1}^D (\dRate{i}+\gamma_i) \alpha_i + \sum_{i=1}^N (\delta \beta_i +  g_i e^{\delta \beta_i} ) \leq L \right\},
  \end{align*}
  where the last equality holds with $L = - \log \epsilon$.
\end{remark}

\begin{remark}\label{rem:only-space-matters}
  Refining along the spatial or the stochastic variables has different
  effects on the error of the MISC
  estimator. %
  Indeed, due to the double exponential
  $e^{- \sum_{j=1}^N g_j \exp(\delta {\beta_j})}$ in
  \eqref{eq:deltaE_ass}, the stochastic contribution to the error will
  quickly fade to zero, which in turn implies that most of the work
  will be used to reduce the deterministic error. This is confirmed
  by the fact that the error convergence rate in Theorem \ref{thm:misc_complexity}
  only depends on $\gamma_i$ and $\dRate{i}$, i.e.,
  the cost and error rates of the deterministic solver, respectively.
  This observation coincides with that in \cite[page 2299]{bieri:sparse.tensor.coll}:
  ``since the stochastic error decreases exponentially, the convergence rate
    should tend towards the algebraic rate of the spatial
    discretization [...]; see Proposition 3.8''.
  Compared with the method proposed in \cite{bieri:sparse.tensor.coll},
  MISC takes greater advantage of this fact, since it is based
  on an optimization procedure, cf. equation \eqref{eq:opt_set};
  this performance improvement is well documented by the comparison
  between the two methods shown in the next section.
  Figure \ref{fig:MISC_set} shows the multi-indices included in
  $\mathcal{I}$ according to \eqref{eq:opt_set_detailed_logform}
  for increasing values of $L$, for a problem with $N=D=1$, $\gamma_i=1,r=2$, and $g=1.5$:
  as expected, the shape of $\mathcal{I}$ becomes more and more curved as $L$ grows, due to this
  lack of balance between the stochastic and deterministic directions.
\begin{figure}[h]
  \centering
  \includegraphics[width=0.5\textwidth]{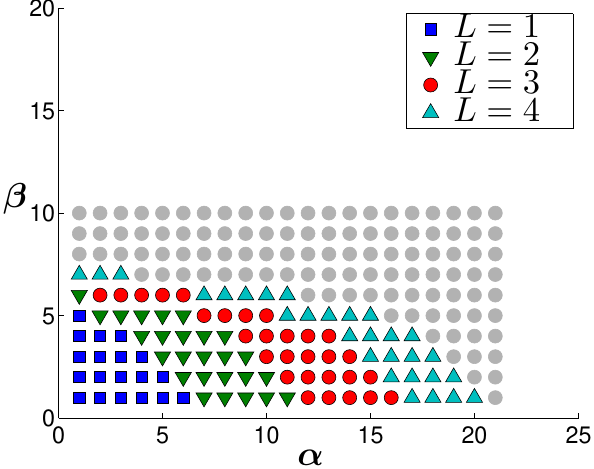}
  \caption{Index sets $\mathcal{I}(L)$ for $D=N=1$, according to
    equation \eqref{eq:opt_set_detailed_logform}.}
  \label{fig:MISC_set}
\end{figure}
\end{remark}

\begin{remark}
  Theorem~\ref{thm:misc_complexity} is only valid in case
  Assumptions~\ref{assump:wellposed}--~\ref{assump:dW_dE_model} are
  true. In the next section we motivate these assumption for a
  specific elliptic problem that we use for numerically testing the
  MISC method. However, we stress that deriving bounds on the error
  and work contributions is problem-dependent and the corresponding
  analysis must be carried out in each case.  Moreover, under a
  different set of assumption the complexity theorem would have to be
  rewritten accordingly.
\end{remark}

\section{Example and numerical evidence}\label{s:numerics}

In this section, we test the effectiveness of the MISC approximation on some instances
of the general elliptic equation \eqref{eq:ell_PDE_strong_yy_example} in Example \ref{example:ell};
more precisely, we consider a problem with one physical dimension ($d=1$)
as well as a more challenging problem with three dimensions ($d=3$);
in both cases, we set $\mathscr{B} = [0,1]^d$, $\forcing(\xx)=1$.
As for the random diffusion coefficient, we set %
\begin{equation}\label{eq:log-uniform-a}
      a(\xx,\yy)= e^{\gamma_N(\xx,\yy)}, \quad
      \gamma_N(\xx,\yy) = \sum_{n=1}^N \lambda_n \psi_n(\xx) y_n,
\end{equation}
where $y_n$ are uniform random variables over $\Gamma_n = [-1,1]$, $\lambda_n = \sqrt{3} \exp(-n)$
and take $\psi_n$ to be a tensorization of trigonometric functions.
More precisely, we define the function
\[
\phi_n(x) =
\begin{cases}
\displaystyle \sin\left(\frac{n}{2} \pi x  \right) & \text{if $n$ is even} \\[8pt]
\displaystyle \cos\left(\frac{n-1}{2} \pi x\right) & \text{if $n$ is odd}
\end{cases}
\]
and set $\psi_n(x) = \phi_n(x)$ if $d=1$. If $d=3$, we take
$\psi_n(\vec x) = \phi_{i(n)}(x_1)\phi_{j(n)}(x_2)\phi_{k(n)}(x_3)$ for some
indices ${i(n), j(n), k(n)}$ detailed in Table~\ref{tab:modes}.
Observe that the boundedness of the supports of the random variables $y_n$
guarantees the existence of the two bounding constants
in equation \eqref{eq:bounding_a}, $a_{min}$ and $a_{max}$,
that in turn assures the well posedness of the problem.
Finally, the quantity of interest is defined as
\begin{equation}\label{eq:qoi_def}
  \qoi(\yy) = \int_\mathscr{B} u(\xx,\yy) Q(\xx) d\xx,
  \quad Q(\xx) = \frac{1}{(\sigma\sqrt{2\pi})^d}\exp \left( -\frac{\|\xx-\xx_0\|_2^2}{2\sigma^2} \right)
\end{equation}
with $\sigma=0.16$ and locations $\xx_0=0.3$ for $d=1$ and
$\xx_0=[0.3, 0.2, 0.6]$ for $d=3$. We also make the choice $h_0=1/3$
in Assumption~\ref{assump:growth_of_dof}. With these values and using
the coarsest discretization, $h_0=1/3$, in all dimensions, the
coefficient of variation of the quantity of interest can be
approximated to be between $90\%$ and $100\%$ depending on the number
of dimensions, $d$, and the number of random variables, $N$, that we
consider below.

\begin{table}[h]
\centering
\begin{tabular}{l|rrrrrrrrrr}
$n$ & 1 & 2 & 3 & 4 & 5 & 6 & 7 & 8 & 9 & 10 \\
\hline
$i(n)$ & 1 & 2 & 1 & 1 & 3 & 2 & 2 & 1 & 1 & 1 \\
$j(n)$ & 1 & 1 & 2 & 1 & 1 & 2 & 1 & 3 & 2 & 1 \\
$k(n)$ & 1 & 1 & 1 & 2 & 1 & 1 & 2 & 1 & 2 & 3 \\
\hline
\end{tabular}
\caption{Included functions for $d=3$ in \eqref{eq:log-uniform-a}. Here
  $\psi_n(\vec x) = \phi_{i(n)}(x_1)\phi_{j(n)}(x_2)\phi_{k(n)}(x_3)$.}\label{tab:modes}
\end{table}

\subsection{Verifying bounds on work and error contributions}
In this subsection we discuss the validity of Assumptions \ref{assump:growth_of_dof}
to \ref{assump:dW_dE_model}, upon which the MISC convergence theorem is based.
To this end, we analyze separately the properties of the deterministic solver and of
the collocation method applied to the problem just introduced.

\paragraph{Deterministic solver}
The deterministic solver considered in this work consists of a tensorized
finite difference solver, with the grid size along each direction,
$x_1,\ldots,x_d$, defined by $h_{i,\alpha_i} = h_0 2^{-\alpha_i}$ and no other discretization parameters are considered:
therefore, $D=d$, Assumption \ref{assump:growth_of_dof} is satisfied, and,
due to the Dirichlet boundary conditions prescribed for $u$,
the overall number of degrees of freedom of the corresponding finite difference solution is
$\prod_{i=1}^D \left( \frac{1}{h_{i,\alpha_i}}-1 \right)
\leq \prod_{i=1}^D \left( \frac{1}{h_{i,\alpha_i}} \right)$.
The associated linear system is solved with the GMRES method. We have numerically fitted the parameters, $\vartheta$ and
  $C_{\mathrm{GMRES}}$, in the model:
\[\work{\qoi_\valpha} \leq C_{\mathrm{GMRES}}  \prod_{i=1}^D
  (h_{i,\alpha_i})^{-\vartheta},\]
for each individual tensor grid solve and found that $\vartheta=1$
gives a good fit in our numerical experiments. From this we can
recover the rates $\{\tilde{\gamma}_i\}_{i=1}^D$ and the constant
$C_{\mathrm{work}}^{\textnormal{det}}$ in \eqref{eq:deltaW_alpha_est}
with the following argument: since computing
$\vec \Delta^\textnormal{det}[\qoi^\valpha]$ requires up to $2^D$
solver calls, each on a different grid (cf. equation
\eqref{eq:delta_det_def}), we have
\begin{align*}
\Delta W_\valpha^\textnormal{det}
& = \work{\vec \Delta^\textnormal{det}[ \qoi_\valpha ]}
= \sum_{\jj \in \{0,1\}^D} \work{\qoi_{\valpha-\jj}} \\
& \leq C_{\mathrm{GMRES}}\sum_{\jj \in \{0,1\}^D} \prod_{i=1}^D \left( h_0 2^{-(\alpha_i-j_i)}\right)^{-\vartheta} \\
& = C_{\mathrm{GMRES}} \left( \prod_{i=1}^D \left( h_0 2^{-\alpha_i}\right)^{-\vartheta} \right) \sum_{\jj \in \{0,1\}^D} \prod_{i=1}^D 2^{-j_i \vartheta} \\
& = C_{\mathrm{GMRES}} (1+ 2^{-\vartheta})^D \prod_{i=1}^D  (h_{i,\alpha_i})^{-\vartheta},
\end{align*}
i.e., bound \eqref{eq:deltaW_alpha_est} is verified with
$\widetilde{\gamma}_i = \vartheta, \forall i=1,\ldots,D$ and
$C_{\mathrm{work}}^{\textnormal{det}} = C_{\mathrm{GMRES}} (1+
2^{-\vartheta})^D$.
Hence, the sum of costs of the solver calls is proportional to the
cost of the call on the finest grid.

Concerning the error contribution $\Delta E_\valpha^\textnormal{det}$, we observe numerically
that bound \eqref{eq:dE_alpha_est} holds true in practice with $\widetilde{\dRate{}}_i=2$,
$i=1,\ldots,D$, due to the fact that $a \in C^{\infty}(\mathscr{B})$
for $\rho$-almost every $\yy \in \Gamma$,
$\forcing \in C^{\infty}(\mathscr{B})$
and the function $Q$ appearing in the quantity of interest \eqref{eq:qoi_def}
is also infinitely differentiable, confined in a small region inside the domain
and zero up to machine precision on the boundary.
In more detail, assuming for a moment that Assumption \ref{assump:dW_dE_factor} is valid
(we will numerically verify it later in this section),
in Figure \ref{fig:check_rates_det} we show the value of
$\Delta E_{\valpha,\vbeta} = \Delta E_\valpha^\textnormal{det} \Delta E_\vbeta^\textnormal{stoc}$ for fixed $\vbeta=\oone$
and variable $\valpha = j \bar{\valpha} + \oone, j=1,2,\ldots$,
as well as the corresponding value of the bound \eqref{eq:dE_alpha_est} for $\Delta E_\valpha^\textnormal{det}$.
The line obtained by choosing $\bar{\valpha} = [1,\, 0,\, 0]$ confirms that the size of $\Delta E_\valpha^\textnormal{det}$
indeed decreases exponentially fast with respect to $\alpha_1$, and by fitting the
computed values of $\Delta E_\valpha^\textnormal{det}$
we obtain that the convergence rate is $\widetilde{\dRate{}}_j=2$, as previously mentioned;
analogous conclusions can be obtained for $\alpha_2$ and $\alpha_3$
by setting $\bar{\valpha} = [0,\, 1,\, 0]$ (shown in Figure \ref{fig:check_rates_det})
and $\bar{\valpha} = [0,\, 0,\, 1]$ (not shown). Most importantly,
confirmation of the product structure of $\Delta E_\valpha^\textnormal{det}$ can be obtained by observing, e.g.,
the decay of $\Delta E_\valpha^\textnormal{det}$ for $\bar{\valpha} = [1,\, 1,\, 0]$ and $\bar{\valpha} = [1,\, 1,\,1]$.

\begin{figure}
  \subfigure[For fixed ${\vbeta=\oone}$ and variable
  ${\valpha = k\bar{\valpha}+\oone}$.\label{fig:check_rates_det}]
  {\includegraphics[page=1,width=0.48\textwidth]{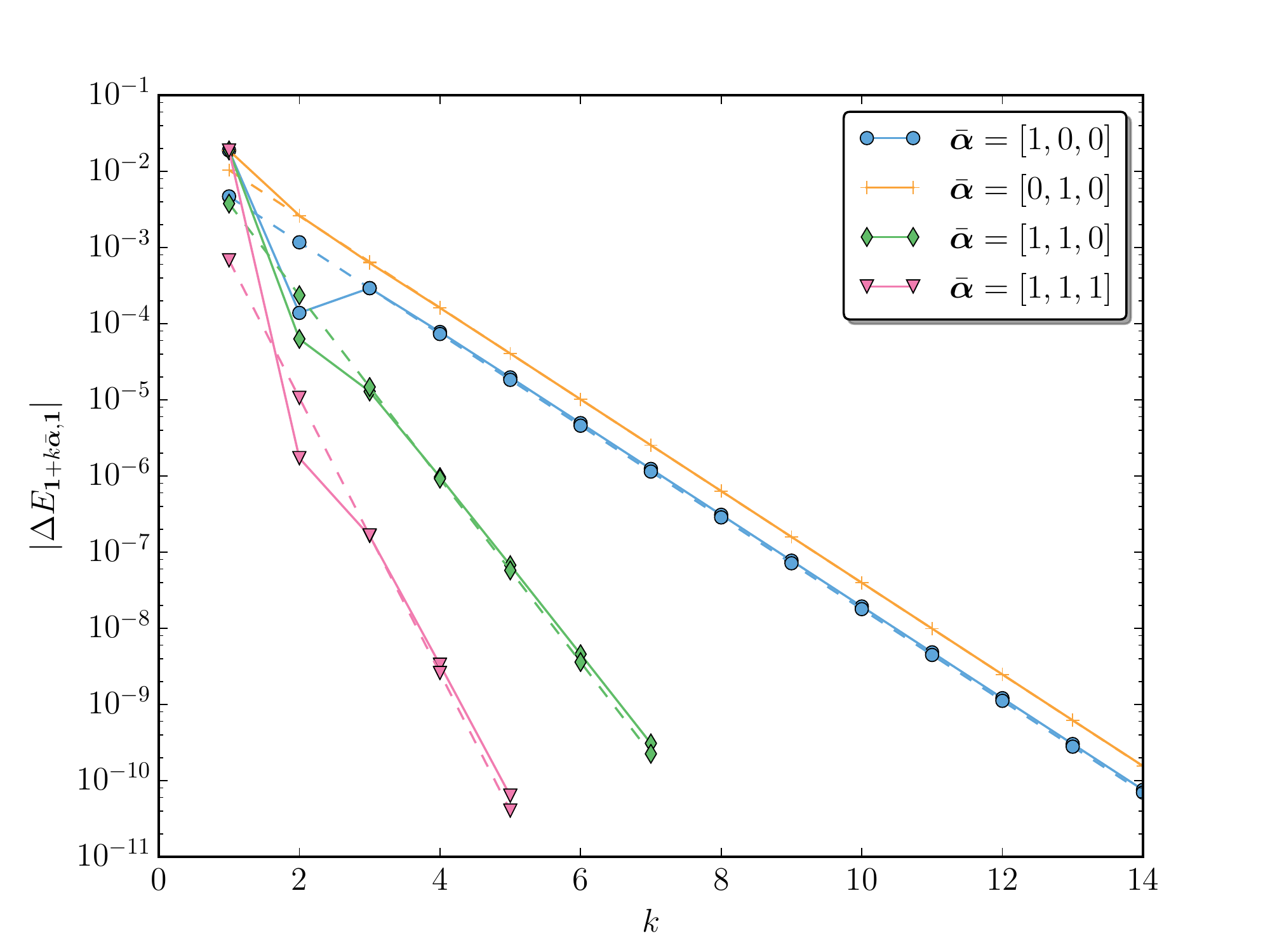}}
  \subfigure[For fixed ${\valpha=\oone}$ and variable
  ${\vbeta = k\bar{\vbeta}+\oone}$.
  \label{fig:check_rates_stoc}]
  {\includegraphics[page=2,width=0.48\textwidth]{COOK_d3_N5_rates.pdf}}
  \caption{Verifying the validity of the bound \eqref{eq:deltaE_ass}
    for the value of $|\Delta E_{\valpha,\vbeta}|$ for the test case
    with $D=3$ and $N=5$.  The {\em dashed} lines are based on the model
    in \eqref{eq:deltaE_ass} with ${\tilde{\dRate{}}_i=2}$ for all
    $i=1,2,3$ and $g_j$ as in Table~\ref{tab:g_values} for
    $j=1,\ldots, 5$. The {\em solid} lines are based on computed
    values. }
\end{figure}

\paragraph{Stochastic discretization} The interpolation over the parameter space is based
on the tensorized Lagrangian interpolation technique with Clenshaw-Curtis
points explained in Section \ref{s:interpolation}, cf. eqs. \eqref{eq:CC_def} and \eqref{eq:m_CC}.
In particular, due to the nestedness of the Clenshaw-Curtis points,
adding the operator $\Delta^\textnormal{stoc}[\descqoi{\valpha}{\vbeta}]$ to the MISC estimator will require
$\Delta W_\vbeta^\textnormal{stoc} = \prod_{j=1}^N \big( m(\beta_j) - m(\beta_j-1) \big)$ new collocation points,
which in view of equation \eqref{eq:m_CC}  can be bounded as
\[   m(\beta_j) - m(\beta_j-1) =
  \begin{cases}
    1 \mbox{ if }  \beta_j=1 \\
    2 \mbox{ if } \beta_j=2 \\
    2^{\beta_j-2}, \mbox{ if } \beta_j>2,
  \end{cases}
  \leq 2^{\beta_j{-1}}, \quad \forall j=1,2,\ldots, %
\]
provided that the set $\mathcal{I}$ is downward closed: Assumption \ref{assump:growth_of_dof}
and bound \eqref{eq:deltaW_beta_est} in Assumption \ref{assump:dW_dE_model} are thus verified.
Observe that the nestedness of the Clenshaw-Curtis knots is a key property here:
indeed, if the nodes are not nested $\Delta W_\vbeta^\textnormal{stoc}$ is not uniquely defined, i.e., it
depends on the set $\mathcal{I}$ to which  $\Delta^\textnormal{stoc}[\descqoi{\valpha}{\vbeta}]$ is being added,
see, e.g., \cite[Example 1 in Section 3]{nobile.eal:optimal-sparse-grids}.

Finally, to discuss the validity of bound \eqref{eq:deltaE_beta_est} for $\Delta E_\vbeta^\textnormal{stoc}$,
we rely on the theory developed in our previous works
\cite{back.nobile.eal:optimal,nobile.eal:optimal-sparse-grids}.
We begin by introducing the Chebyshev polynomials of the first kind
$\ortp_{q}(y)$ for $y \in [-1,1]$, which are defined by the relation
\[\ortp_{q}(\cos(\theta))
= \cos(q \theta), \quad 0 \leq \theta \leq \pi, \quad q \in \nset.\]
Then, for any multi-index $\qq \in \nset^N$, we consider the
$N$-variate Chebyshev polynomials $\oortp_\qq(\yy) = \prod_{n=1}^N \ortp_{q_n}(y_n)$
and introduce the spectral expansion of $f:[-1,1]^N\rightarrow \rset$
over  $\{\oortp_\qq\}_{\qq \in \nset^N}$,
\[%
f(\yy) = \sum_{\qq \in \nset^N} f_\qq \oortp_\qq(\yy),
\quad
f_\qq = \int_\Gamma f(\yy) \oortp_\qq(\yy) \prod_{n=1}^N
\frac{1}{\sqrt{1-y_n^2}} d\yy,
\]
Next, given any $\ellrad_{1},\ellrad_{2},\ldots,\ellrad_{N}>1$
we introduce the Bernstein polyellipse
$\mcE_{\ellrad_1,\ldots,\ellrad_N} =\prod_{n=1}^N \mcE_{n,\ellrad_n}$,
where $\mcE_{n,\ellrad_n}$ denotes the ellipses in the complex plane defined as
\begin{align*}%
  \mcE_{n,\ellrad_n} = \bigg\{ z_n \in \mathbb{C} : \,\,
  \real{z} \leq \frac{\ellrad_{n} + \ellrad_{n}^{-1}}{2}\cos{\phi}, \,\,\,
  \im{z} \leq \frac{\ellrad_{n} - \ellrad_{n}^{-1}}{2}\sin{\phi},
  \,\, \phi \in [0, 2\pi) \bigg\},
\end{align*}
and recall the following lemma (see \cite[Lemma 2]{nobile.eal:optimal-sparse-grids} for a proof).

\begin{lemma}\label{lemma:inc_cheb_conv}
  Let $f: [-1,1]^N \rightarrow \rset$, and
  assume that there exist $\ellrad_{1},\ellrad_{2},\ldots,\ellrad_{N}>1$
  such that $f$ admits a complex continuation $f^*:\mathbb{C}^N\rightarrow \rset$
  holomorphic in the Bernstein polyellipse $\mcE_{\ellrad_1,\ldots,\ellrad_N}$
  with $\sup_{\zz \in \mcE_{\ellrad_1,\ldots,\ellrad_N}} |f^*(\zz)| \leq B$ and
  $B = B(\ellrad_{1},\ellrad_{2},\ldots, \ellrad_{n}) < \infty$.
  Then $f$ admits a Chebyshev expansion %
  that converges in $C^0([-1,1]^N)$, and whose coefficients $f_\qq$ are such that
  \begin{equation}\label{eq:polyell_cheb_conv}
    | f_\qq | \leq C_\textnormal{Cheb}(\qq) \prod_{n=1}^N e^{-g^*_n q_n},  \quad g^*_n = \log \ellrad_{n}
  \end{equation}
  with $C_\textnormal{Cheb}(\qq) = 2^{\|\qq \|_0} B$, where $\|\qq \|_0$ denotes the
  number of non-zero elements of $\qq$.
\end{lemma}

The following lemma then shows that the region of analyticity of $\qoi(\yy)$ indeed contains
a Bernstein ellipse, so that a decay of exponential type can be expected for its
Chebyshev coefficients.

\begin{lemma}\label{lemma:qoi_polyellipse}
The quantity of interest $\qoi(\yy)=\Theta[u(\cdot,\yy)]$ is analytic in a Bernstein polyellipse
with parameters $\ellrad_n = \tau_n + \sqrt{\tau_n^2+1}$, for any $\tau_n < \frac{\pi}{2 N \lambda_n}$.
\end{lemma}
\begin{proof}
Equation \eqref{eq:ell_PDE_strong_yy_example} can be extended in the complex domain
by replacing $\yy$ with $\zz \in \mathbb{C}^N$, and is analytic in the set
$\Sigma = \left\{ \zz \in \mathbb{C}^N : \Re e{[a(\xx,\zz)]} >0 \right\}$,
see, e.g., \cite{babuska.nobile.eal:stochastic2}. By writing $z_n = b_n + i c_n$, we have
\begin{align*}
a(\xx,\zz)
& = \exp \left( \sum_n z_n \lambda_n \psi_n(\xx) \right)
  = \exp \left( \sum_n b_n \lambda_n \psi_n(\xx) \right) \exp \left( \sum_n i c_n \lambda_n \psi_n(\xx) \right) \\
& = \exp \left( \sum_n b_n \lambda_n \psi_n(\xx) \right)
	\left[ \cos \left(\sum_n c_n \lambda_n \psi_n(\xx) \right) + i \sin \left( \sum_n c_n \lambda_n \psi_n(\xx) \right) \right]
\end{align*}
so that the region $\Sigma$ can be rewritten as
\[
\Sigma =
\left\{ \zz = \bb + i \cc \in \mathbb{C}^N : \cos \left( \sum_n c_n \lambda_n \psi_n(\xx) \right) >0 , \forall \xx \in \mathscr{B} \right\}.
\]
Such a region includes the smaller region
\[\Sigma_2 = \left\{ \zz = \bb + i \cc \in \mathbb{C}^N :
\left\| \sum_n c_n \lambda_n \psi_n \right\|_{L^\infty(\mathscr{B})} < \frac{\pi}{2} \right\},\]
which in turn includes
\[\Sigma_3
= \left\{ \zz = \bb + i \cc \in \mathbb{C}^N :
\sum_n  \lambda_n |c_n|  < \frac{\pi}{2} \right\},\]
where the last equality is due to the fact that, by construction,
$\left\|\psi_n\right\|_{L^\infty(\mathscr{B})}=1$, cf. equation \eqref{eq:log-uniform-a}.
Next we let $\tau_n = \frac{\pi}{2 N \lambda_n}$
and define the following subregion of $\Sigma_3$:
\[\Sigma_4 = \bigg\{ \zz = \bb + i \cc \in \mathbb{C}^N : |c_n|  < \tau_n \bigg\} \subset \Sigma_3.
\]
$\Sigma_4$ is actually a polystrip in the complex plain that it in turn contains
the Bernstein ellipse with parameters $\ellrad_n$ such that
\[\frac{\ellrad_n - \ellrad_n^{-1}}{2} = \tau_n
\Rightarrow \ellrad_n^2 - 1 - 2\tau_n\ellrad_n = 0
\Rightarrow \ellrad_n = \tau_n + \sqrt{\tau_n^2+1}
\]
in which $u(\xx,\yy)$ is analytic.
Finally, the quantity of interest, $\qoi=\Theta[u]$, is also analytic in the same
Bernstein polyellipse due to the linearity of the operator $\Theta$.
\end{proof}
\begin{remark}
  Incidentally, we remark that the choice of $\tau_n$ considered in Lemma \ref{lemma:qoi_polyellipse}
  degenerates for $N \to \infty$. In this case, if we know that
  $\sum_{n=0}^\infty (\lambda_n \|\psi_n\|_{L^\infty(\mathscr{B})})^p  < \infty$
  for some $p<1$, then we could set
  $\tau_n = \frac{\pi}{2} (\lambda_n \|\psi_n\|_{L^\infty(\mathscr{B})})^{p-1}$,
  which does not depend on $N$.
\end{remark}

\begin{lemma}
  \begin{equation}\label{eq:deltaE_beta_gen}
    \Delta E_\vbeta^\textnormal{stoc} \leq C_E e^{-\sum_{n=1}^N g^*_n m(\beta_n-1)} \mathbb{M}^{m(\vbeta)}
  \end{equation}
  holds, where $C_E= 4^N B \prod_{n=1}^N \frac{1}{1-e^{-g^*_n}}$, $B$ as in Lemma \ref{lemma:inc_cheb_conv},
  $g^*_n = \log \ellrad_n $ with $\ellrad_n$ as in Lemma \ref{lemma:qoi_polyellipse},
  and $\mathbb{M}^{m(\vbeta)}$ has been defined in equation \eqref{eq:leb_def}.
\end{lemma}
\begin{proof}
  Combining Lemmas \ref{lemma:inc_cheb_conv} and \ref{lemma:qoi_polyellipse},
  we obtain that the Chebyshev coefficients of $\qoi$ can be bounded as
  \[| \qoi_\qq | \leq C_\textnormal{Cheb}(\qq) \prod_{n=1}^N e^{-g^*_n q_n},\]
  with $g^*_n = \log \ellrad_n = \log (\tau_n + \sqrt{\tau_n^2+1})$ and $\tau_n$
  as in Lemma \ref{lemma:qoi_polyellipse}. Then, the result can
  be obtained following the same argument of
  \cite[Lemma 5]{nobile.eal:optimal-sparse-grids}.
\end{proof}
To conclude, we first observe that $\mathbb{M}^{m(\vbeta)}$
grows logarithmically with respect to $m(\vbeta)$, see eq. \eqref{eq:CC_leb_est},
so it is asymptotically negligible in the estimate above, i.e. we can write
  \[\Delta E_\vbeta^\textnormal{stoc} \leq C_{E_2}(\epsilon) \prod_{n=1}^N e^{-g^*_n (1-\epsilon_E) m(\beta_n-1)}\]
for an arbitrary $\epsilon_E > 0$ and with $C_{E_2}(\epsilon_E)>C_E$, and furthermore that
the definition of $m(i)$ in \eqref{eq:m_CC} implies that $m(i-1) \geq \frac{m(i)-1}{2}$.
We can finally write
\begin{align*}
  \Delta E_\vbeta^\textnormal{stoc}
  & \leq C_{E_2}(\epsilon) \prod_{n=1}^N e^{-g^*_n (1-\epsilon_E) \frac{m(\beta_n)-1}{2}} %
   = C_{\mathrm{error}}^{\textnormal{stoc}} \prod_{n=1}^N e^{-\widetilde{g}_n m(\beta_n)}, %
\end{align*}
with $C_{\mathrm{error}}^{\textnormal{stoc}} = C(\epsilon) \prod_{n=1}^N e^{\frac{g^*_n}{2} (1-\epsilon)}$ and $\widetilde{g}_n = \frac{g^*_n}{2} (1-\epsilon_E)$.
The latter bound actually shows that
bound \eqref{eq:deltaE_beta_est} in Assumption \ref{assump:dW_dE_model} is valid for the test we are considering.
Finally, we point out that in practice we work with the expression \eqref{eq:deltaE_ass},
whose rates $g_n$ are actually better estimated numerically,
using the same procedure used to obtain the deterministic rates
$\tilde {\dRate{}}_j=2$:
we choose a sufficiently fine spatial resolution level $\valpha$, consider a variable $\vbeta = j\bar{\vbeta} + \oone$
and fit the (simplified) model $\Delta E_\vbeta^\textnormal{stoc} \leq C \prod_{n=1}^N e^{-g_n 2^{\beta_n}}$.
The values obtained are reported in Table \ref{tab:g_values}, and they are found
to be equal for the case $d=1$ and $d=3$
(see also \cite{back.nobile.eal:comparison,back.nobile.eal:optimal,nobile.tempone.eal:aniso}).
To make sure that the estimated value of $g_n$ does not depend on the spatial discretization,
one could repeat the procedure for a few different values of $\valpha$ and verify that the estimate
is robust with respect to the spatial discretization: we note, however, that a rough estimate of $g_n$
will also be sufficient, since the convergence of the method is in practice dictated by the deterministic
solver, as we have already discussed in Remark \ref{rem:only-space-matters}.
Figure \ref{fig:check_rates_stoc} then shows the validity of the bound
$\Delta E_\vbeta^\textnormal{stoc} \leq C \prod_{n=1}^N e^{-g_n 2^{\beta_n}}$
comparing for fixed $\valpha=\oone$ and $\vbeta = j\bar{\vbeta} + \oone$
the value of $\Delta E_{\valpha}^\textnormal{det} \Delta E_\vbeta^\textnormal{stoc}$ and the corresponding estimate.

\begin{table}[tbp]
  \centering
  \begin{tabular}{ccccccccccc}
    \hline
	& $g_1$	& $g_2$	& $g_3$	& $g_4$	& $g_5$	& $g_6$	& $g_7$	& $g_8$	& $g_9$	& $g_{10}$	\\
    \hline
	& 2.4855& 2.8174& 4.5044&4.1938 &4.7459 &6.8444	& 7.1513& 7.8622& 8.6584& 9.4545	\\
    \hline
  \end{tabular}
  \caption{Values of rates $g$ for the test cases considered.}
  \label{tab:g_values}
\end{table}

\paragraph{Stochastic-deterministic product structure}
We conclude this section by verifying Assumption \ref{assump:dW_dE_factor}, i.e., the fact that the error
contribution can be factorized as $\Delta E_{\valpha, \vbeta} = \Delta E_\valpha^\textnormal{det} \Delta E_\vbeta^\textnormal{stoc}$ and
that an analogous decomposition holds for $\Delta W_{\valpha,\vbeta}$. While the latter is trivial,
to verify the former we employ the same strategy used to verify the models
for $\Delta E_\valpha^\textnormal{det}$ and  $\Delta E_\vbeta^\textnormal{stoc}$, this time letting both $\valpha$
and $\vbeta$ change for every point,
i.e., $\valpha = j \bar{\valpha} + \oone$ and  $\vbeta = j \vbeta_0 + \oone$.
Figure \ref{fig:check_mixed_rates} shows the comparison between the computed value of
$\Delta E_{\valpha,\vbeta}$ and their estimated counterpart and confirms the validity
of the product structure assumption.

\begin{figure}
  \centering
  \includegraphics[page=3,width=0.5\textwidth]{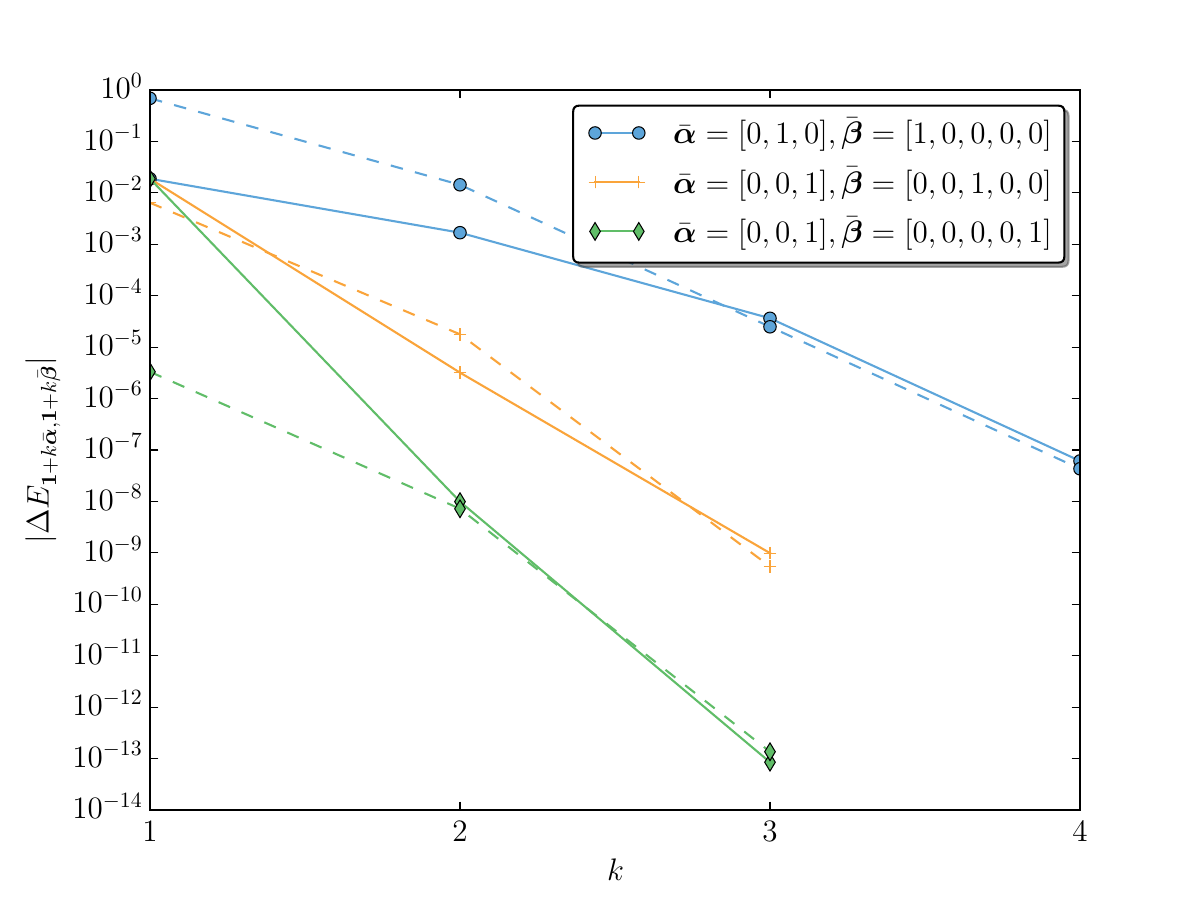}
  \caption{Comparison of $|\Delta E_{\valpha,\vbeta}|$ for
    $\vbeta=j\bar{\vbeta}+\oone$ and $\valpha = j\bar{\valpha}+\oone$
    for the test case with $D=3$ and $N=5$.
    The {\em dashed} lines are based on the model
    in \eqref{eq:deltaE_ass} with ${\tilde{\dRate{}}_i=2}$ for all
    $i=1,2,3$ and $g_j$ as in Table~\ref{tab:g_values} for
    $j=1,\ldots, 5$. The {\em solid} lines are based on computed
    values.}
  \label{fig:check_mixed_rates}
\end{figure}

\subsection{Test setup}

In our numerical tests, we compare MISC with the methods listed below. For each of them
we show (for both test cases considered)
plots of the convergence of the error in the computation of $\E{\qoi}$
with respect to the computational work, taking as a reference value the result obtained
using a well-resolved MISC solution. To avoid discrepancies in running time due to
implementation details, the computational work is estimated in terms of
the total number of degrees of freedom,
i.e., using \eqref{eq:work_decomp} and \eqref{eq:deltaW_ass}. The names used here for the methods
are also used in the legends of the figures showing the convergence plots.

\begin{description}

\item[``a-priori'' MISC] refers to the MISC method with index
  set $\mathcal I$ defined by \eqref{eq:opt_set}, where
  $\Delta W_{\valpha, \vbeta}$ and $\Delta E_{\valpha, \vbeta}$ are
  taken to equal their upper bounds in \eqref{eq:deltaW_ass} and
  \eqref{eq:deltaE_ass}, respectively. The resulting set is explicitly
  written in~\eqref{eq:opt_set_detailed_logform}.
  The convergence rate of this set is predicted by Theorem \ref{thm:misc_complexity},
  cf. Remark \ref{rem:this-is-the-set}. Note that we do not need to determine
  the value of the constants $\CWork$ and $\CError$ since they can be
  absorbed in the parameter $\epsilon$ in \eqref{eq:opt_set}.

\item[``a-posteriori'' MISC] refers to the MISC method with
  index set $\mathcal I$ defined by \eqref{eq:opt_set}, where
  $\Delta W_{\valpha, \vbeta}$ is taken to equal its upper bound in
  \eqref{eq:deltaW_ass}, and $\Delta E_{\valpha, \vbeta}$ is instead
  computed explicitly as $\left| \Delta[\descqoi{\alpha}{\beta}] \right|$.
  Notice that this method is not practical since the cost
  of constructing set $\mathcal I$
  would dominate the cost of the MISC estimator by far. However, this
  method would produce the best possible convergence and serve as a
  benchmark for both ``a-priori'' MISC and the bound \eqref{eq:deltaE_ass}.

\item[MLSC] (only in the case $d>1$) refers to the Multilevel Stochastic Collocation
  obtained by setting $\alpha_1=\ldots=\alpha_D$ (i.e. considering the mesh-size
  as the only discretization parameter), as already mentioned in Remark \ref{remark:webster_and_bieri};
  we recall this is not exactly the MLSC method that was implemented in
  \cite{teckentrup.etal:MLSC}, see again Remark \ref{remark:webster_and_bieri}.
  Just as with MISC, we consider both the ``a-priori'' and ``a-posteriori'' version of MLSC,
  where $\Delta E_{\valpha, \vbeta}$ is taken to be equal to its upper \eqref{eq:deltaE_ass}
  in the former case and assessed by direct numerical evaluation in the latter case.

\item[SCC] refers to the ``Sparse Composite Collocation method''
  in Remark \ref{remark:webster_and_bieri}, see equation \eqref{eq:iso-MISC}.

\item[MIMC] refers to the Multi-Index Monte Carlo method as detailed in
  \cite{abdullatif.etal:MultiIndexMC}, for which the complexity
  $\Order{W_{\max}^{-0.5}}$ can be estimated
   for the test case at hand and as long as $d<4$.

\item[SGSC] refers to the quasi-optimal Sparse Grids Stochastic Collocation (SGSC)
  with fixed spatial discretization as proposed in
  \cite{back.nobile.eal:optimal,nobile.eal:optimal-sparse-grids}.  To
  determine the needed spatial discretization for a given work and for a
  fair comparison against MISC, we actually compute the convergence
  curves of SGSC for all relevant levels of spatial discretizations
  and then show in the plots only the lower envelope of the
  corresponding convergence curves, ignoring the possible
  spurious reductions of error that might happen due to
  non-asymptotic, unpredictable cancellations, cf. Figure \ref{fig:envelope}.
  In this way, we ensure that the
  error shown for such ``single-level methods'' has been obtained with
  the smallest computational error possible.
  Again, this is not a computationally practical method but is taken as a reference
  for what a sparse grids Stochastic Collocation method with optimal balancing
  of the space and stochastic discretization errors could achieve.
\end{description}

\begin{figure}[htp]
  \centering
  \includegraphics[page=1,width=0.7\textwidth]{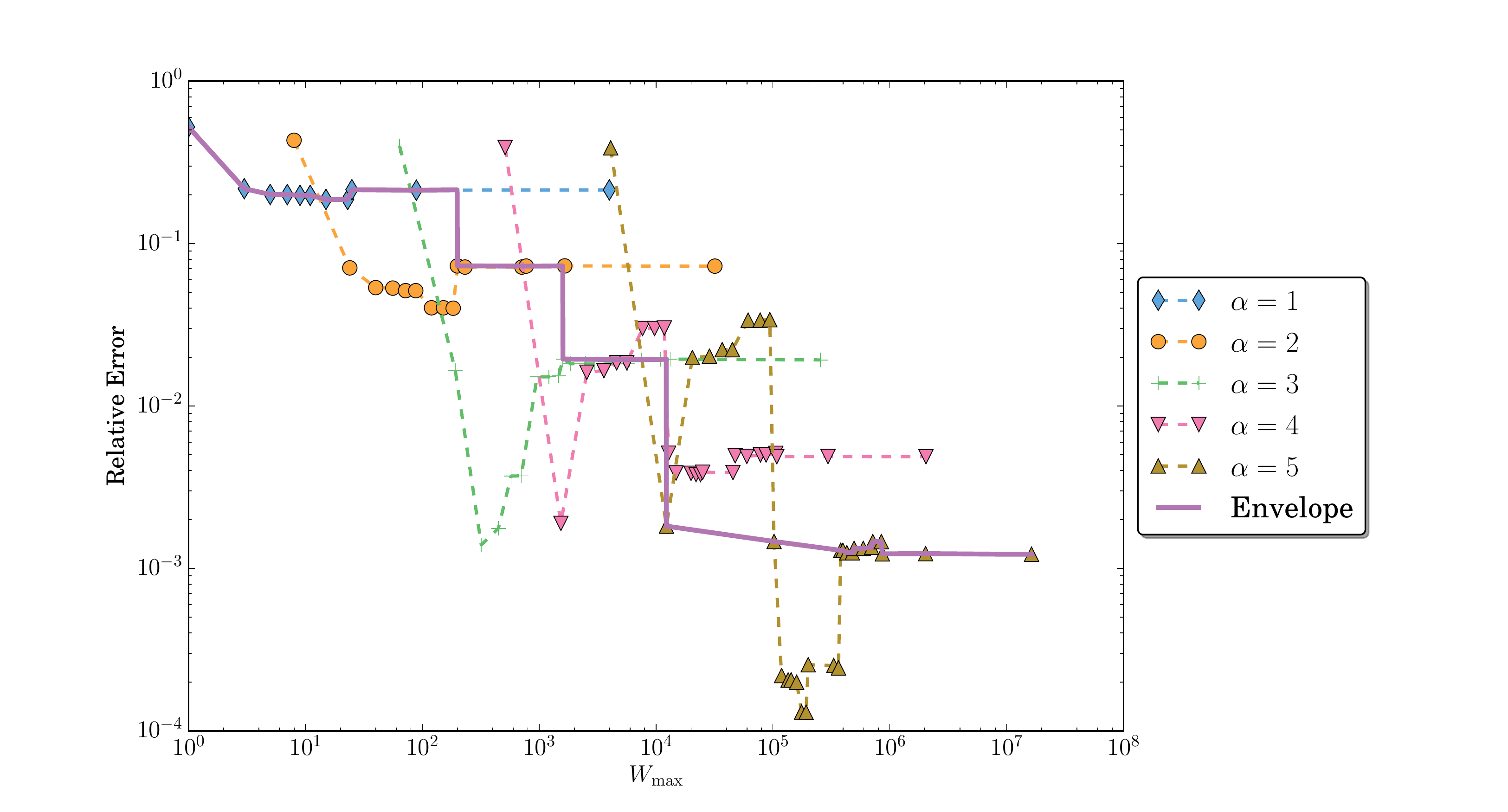}
  \caption{Envelope of SGSC convergence curves for the test case with $d=3$ and $N=10$.}
  \label{fig:envelope}
\end{figure}

\subsection{Implementation details}
\label{ss:impl}
To implement MISC, we need two components:
\begin{enumerate}
\item Given a profit level parameter, $\epsilon$, we build the quasi-optimal set
  $\mathcal I$ based on~\eqref{eq:opt_set}, \eqref{eq:deltaW_ass}
  and~\eqref{eq:deltaE_ass}. One method to achieve this is
  to exploit the fact that this set is downward closed and use the
  following recursive algorithm.
  \begin{verbatim}
FUNCTION BuildSet(epsilon, multiIndex)
    FOR i = 1 to (D+N)
        IF Profit(multiIndex + e_i) > epsilon
        THEN
            ADD multiIndex+e_i to FinalSet
            CALL BuildSet(epsilon, multiIndex+e_i)
        END IF
    END FOR
END FUNCTION
  \end{verbatim}
\item Given the set, $\mathcal I(L)$, we
  evaluate \eqref{eq:misc_estimator}. Here we have two choices:
\begin{itemize}
\item Evaluate the individual terms
  $\vec \Delta[\descqoi{\valpha}{\vbeta}]$ for every
  $\valpha,\vbeta \in \mathcal I$. To do so, we use the operator
  defined in~\eqref{eq:misc_delta_operations} along each stochastic
  and spatial direction. By storing the values of these terms,
  we can evaluate the MISC with different index sets (contained in
  $\mathcal I(L)$), which might be required to test the convergence of
  the MISC method.  Moreover, this implementation is suitable for
  adaptive methods that expand the index set based on some criteria
  and reevaluate the MISC estimator. On the other hand, this
  implementation has a computational overhead since most computed
  values of $\descqoi{\valpha}{\vbeta}$ will actually not contribute
  to the final value of the estimator. However, this computational overhead
  is only a fraction of the minimum time required to evaluate the
  estimator.
\item Use the combination form of~\eqref{eq:misc_estimator} and
  only compute the terms that have $ c_{\valpha,\vbeta} \neq 0$. This would
  remove the overhead of computing terms that make zero contribution
  to the estimator. This implementation is more efficient but less
  flexible as we cannot evaluate the estimator on sets contained
  in $\mathcal I(L)$ or build the set adaptively.
\end{itemize}
\end{enumerate}

\subsection{Test with $D=1$}

Here we consider three different numbers of stochastic variables,
namely $N=1,5,10$. Results are shown in Figure \ref{fig:res_D=1}.
As expected, a-posteriori MISC shows the best convergence, with a-priori
MISC being slightly worse and the single level methods following.
Finally, we verify the accuracy of the estimated asymptotic convergence rate provided by
Theorem \ref{thm:misc_complexity}: in this case,
$\zeta = \frac{\dRate{}_1}{\gamma_1} = \frac{\widetilde{\dRate{}}_1 \log 2}{\widetilde{\gamma}_1 \log 2} = 2$
and $\mathfrak{z}=1$ holds. Hence, the predicted convergence
rate is $W^{-\zeta}(\log W)^{(\zeta +1)(\mathfrak{z}-1)} = W^{-2}$, which appears to be in
good agreement with the experimental convergence rate.

\begin{figure}[htbp]
  \centering
  \includegraphics[page=2, width=0.8\textwidth]{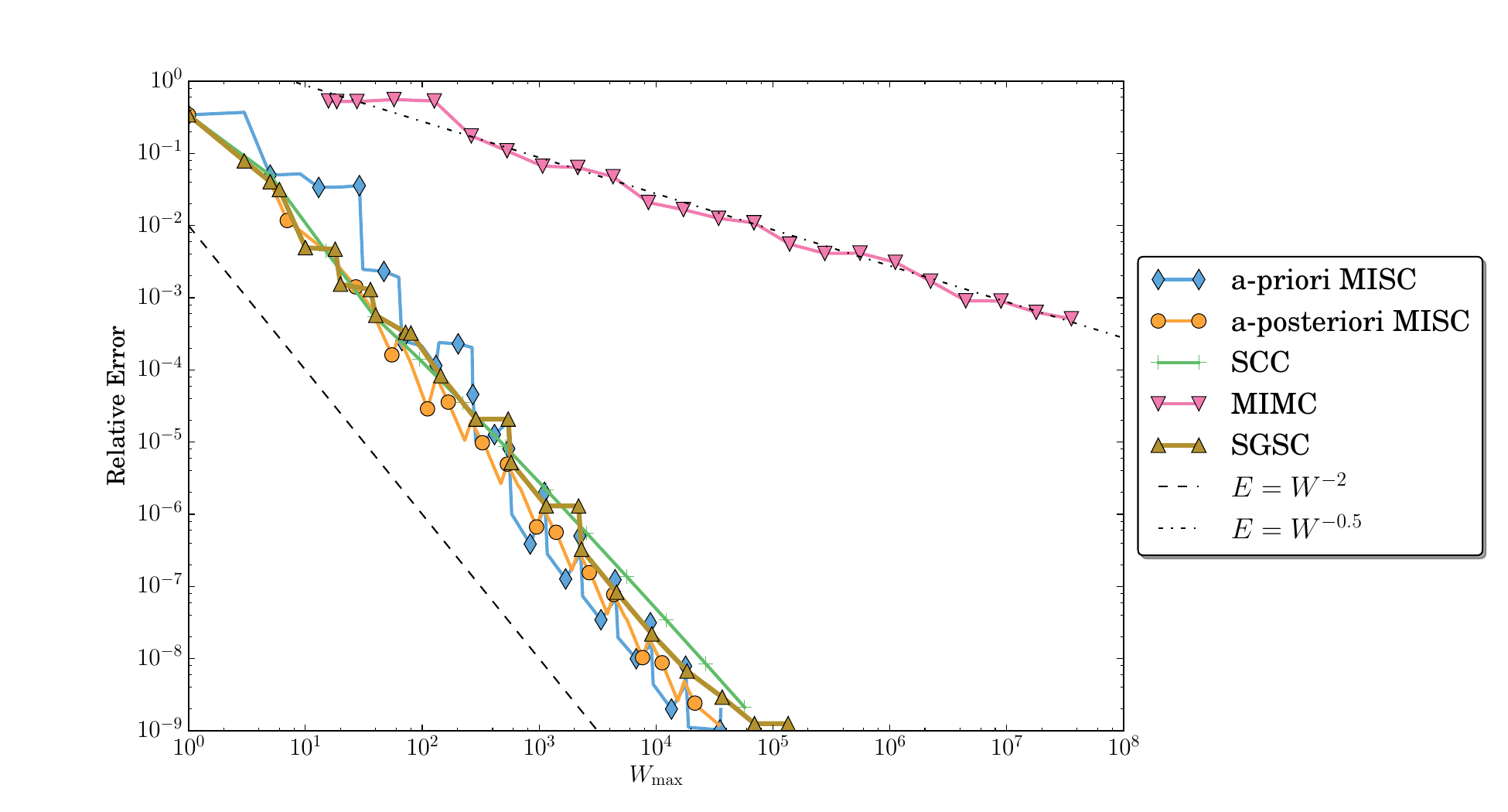}\\
  \includegraphics[page=2, width=0.8\textwidth]{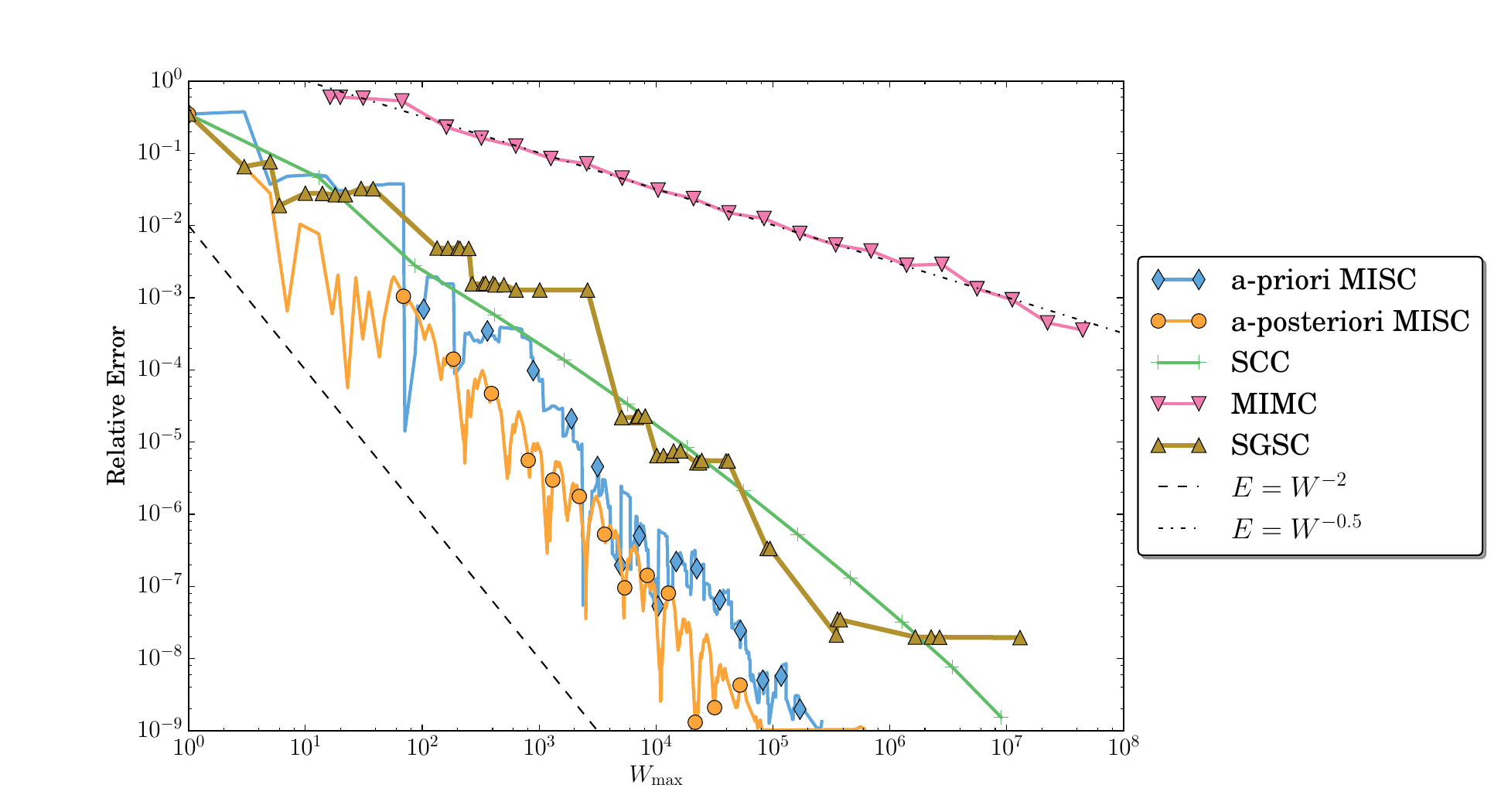}\\
  \includegraphics[page=2, width=0.8\textwidth]{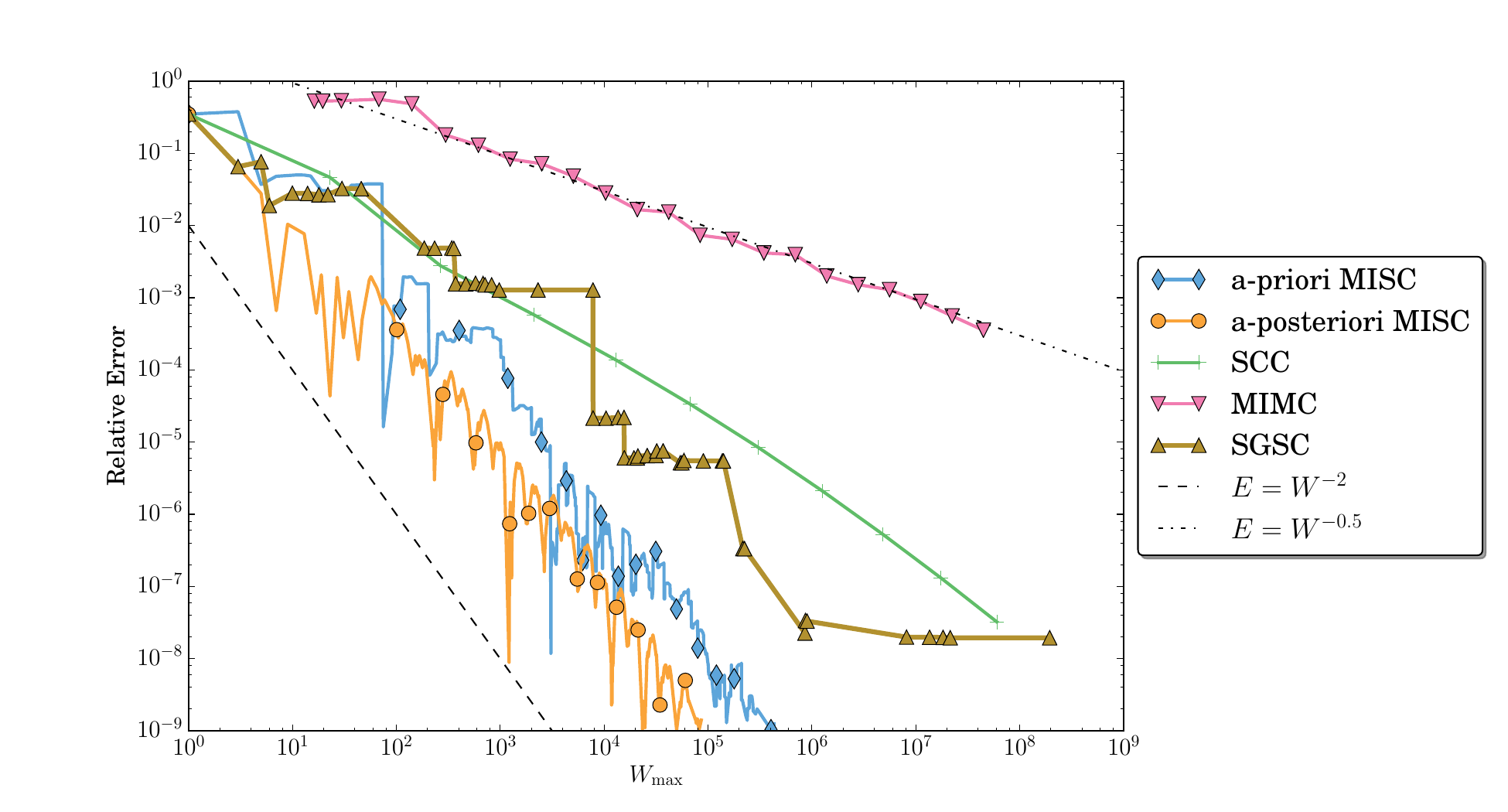}\\
  \caption{Results for test $D=1$, case $N=1$ (top), $N=5$ (center) and $N=10$ (bottom).}
  \label{fig:res_D=1}
\end{figure}

\subsection{Test with $D=3$}
In this case, we obtain the convergence curves shown in
Figure \ref{fig:res_D=3}, where the Multilevel
Stochastic Collocation method has also been included. The hierarchy between the
methods is in agreement with the case $d=1$, with the Multilevel
Stochastic Collocation being comparable or slightly better than single
level methods, but worse than the MISC approaches as expected.

Concerning the accuracy of the theoretical estimate:
since  for this test $\widetilde{\dRate{}}_1 = \widetilde{\dRate{}}_2 = \widetilde{\dRate{}}_2 = 2$
and $\widetilde{\gamma}_1 = \widetilde{\gamma}_2 = \widetilde{\gamma}_3=1$,
$\zeta = 2$ still holds, while this time $\mathfrak{z}=3$; hence, the predicted convergence
rate is $W^{-\zeta}(\log W)^{(\zeta +1)(\mathfrak{z}-1)} = W^{-2} (\log W)^6$.
The plots suggest that the theoretical estimates might
be slightly too optimistic when $N$ increases
but it is important to recall that Theorem \ref{thm:misc_complexity} gives only an asymptotic
result, and the plot could be negatively influenced by pre-asymptotic effects.
Observe also that in this case there are a few data points where
a-posteriori MISC is not better than a-priori MISC; this observation
can be ascribed to the fact that a-posteriori MISC is optimal only with
respect to the upper bound in \eqref{eq:error-decomp}.  In other
words, a-posteriori MISC selects the contributions according to the
absolute value of the contributions but then the MISC estimator is
computed by summing signed contributions.  Hence, cancellations between
contributions with similar sizes and opposite signs will occur.

Finally, we remark that, in our calculations, MLSC and SGSC were not able to
achieve very small errors, unlike MISC. This is due to a limitation in
the linear solver we are using that allows systems with only
up to $2^{17}$ degrees of freedom (around 1GB of memory) to be solved. These
``single-level'' methods hit that limit sooner than MISC since they
entail solving a very large system that comes from isotropically discretizing
all three spatial dimensions.

\begin{figure}[htbp]
  \centering
  \includegraphics[page=2,width=0.8\textwidth]{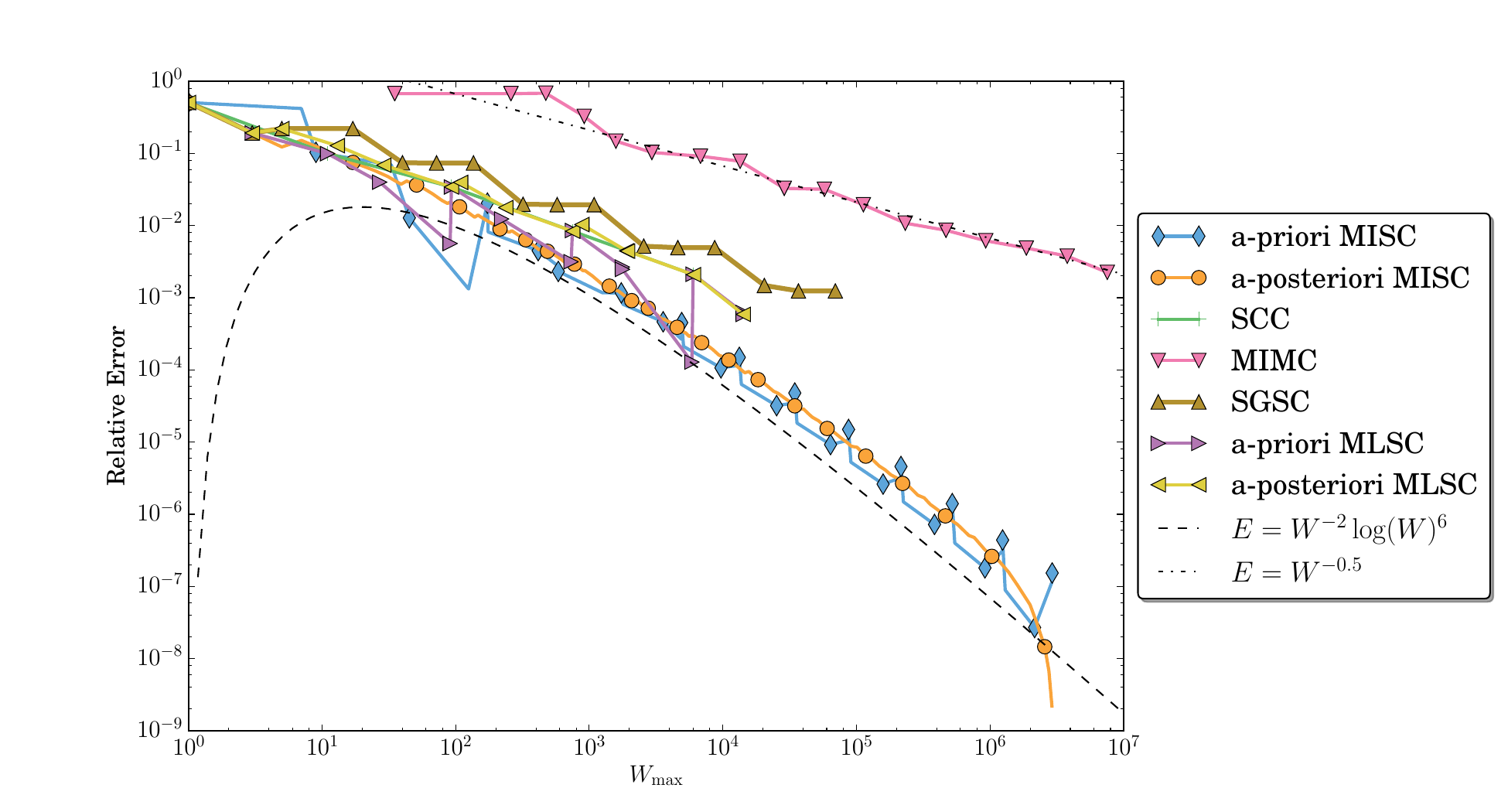}\\
  \includegraphics[page=2,width=0.8\textwidth]{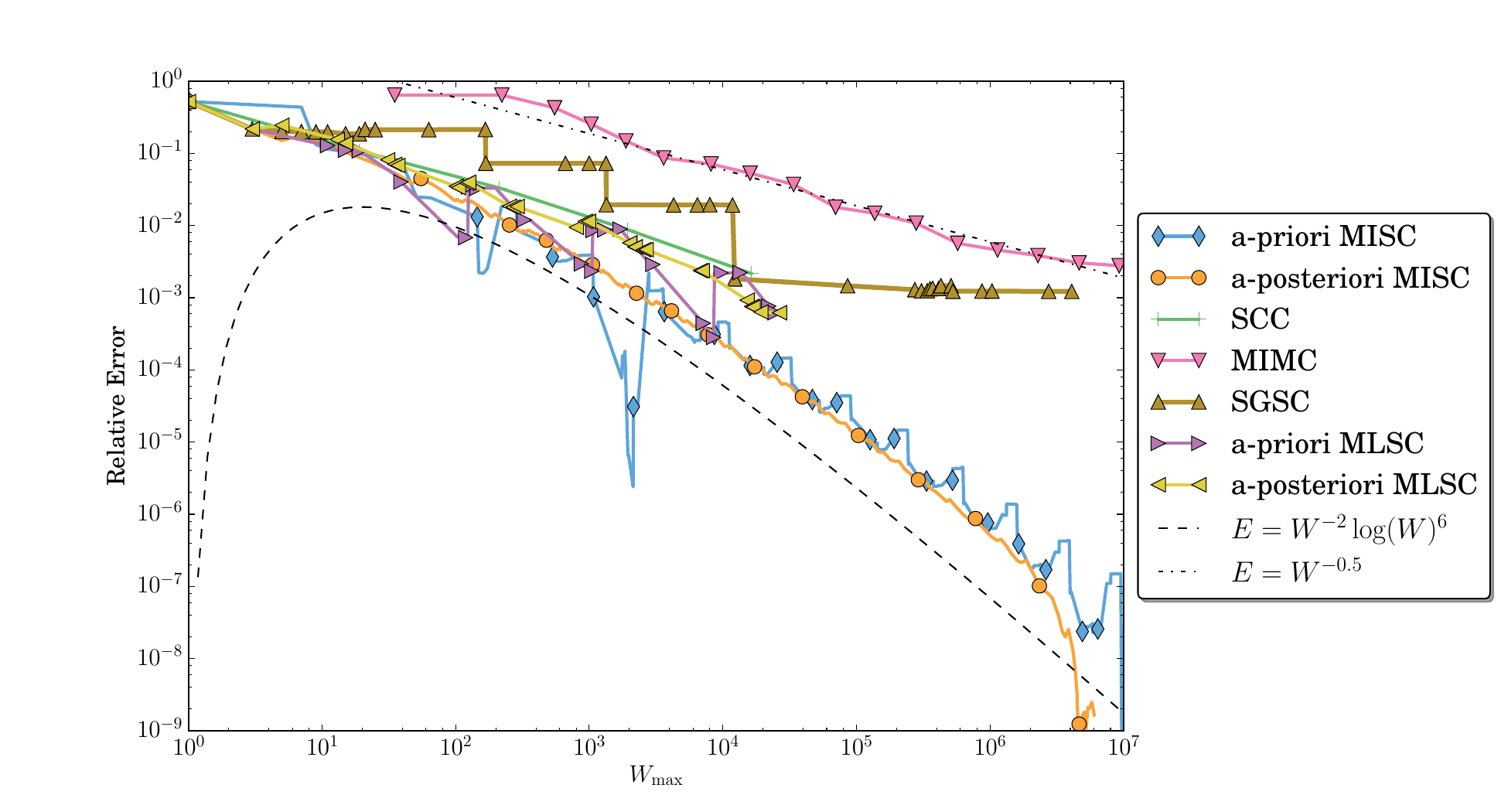}\\
  \includegraphics[page=2,width=0.8\textwidth]{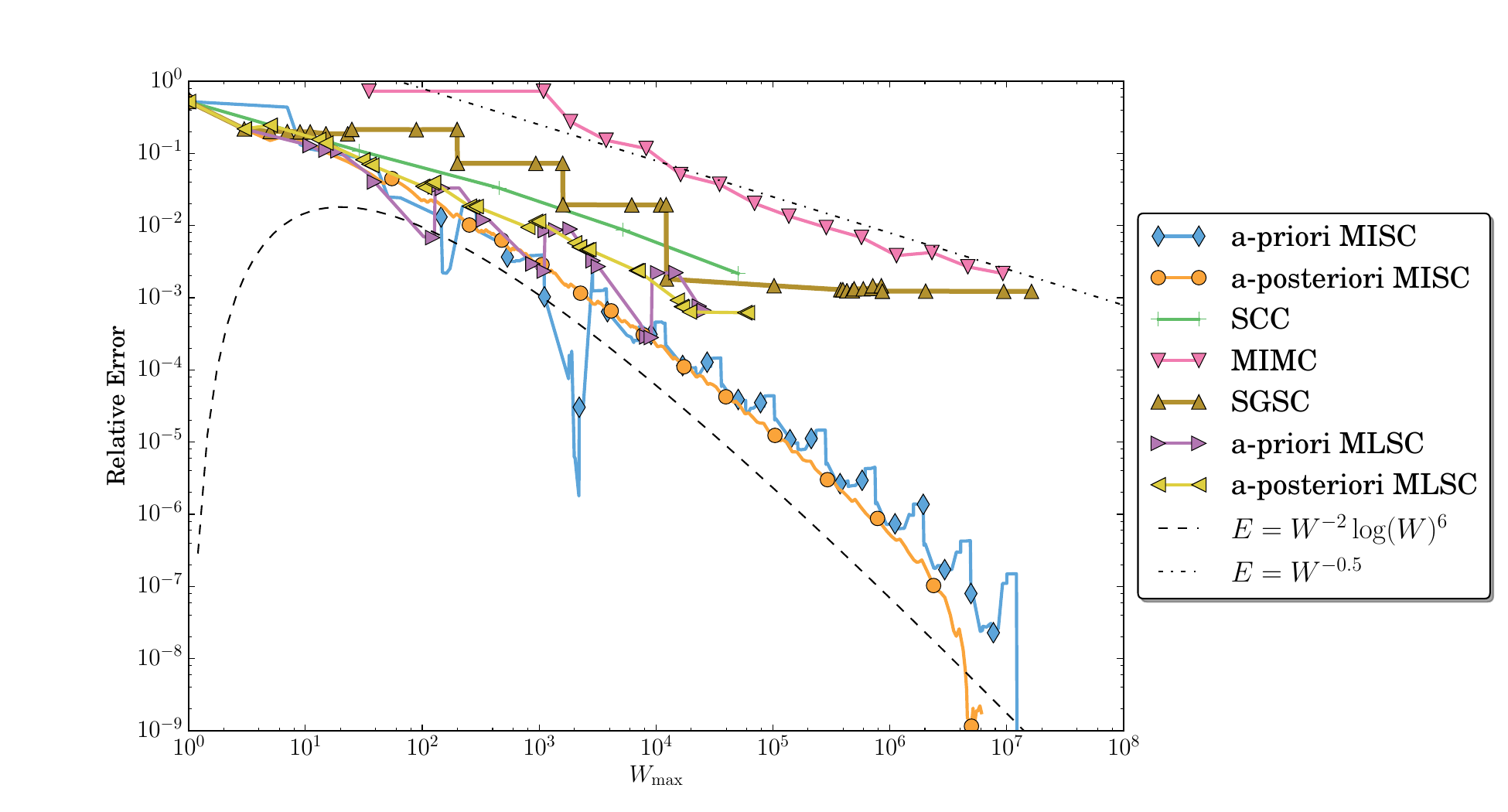}
  \caption{Results for test $D=3$, case $N=1$ (top), $N=5$ (center)
    and $N=10$ (bottom). }
  \label{fig:res_D=3}
\end{figure}

\section{Conclusions}\label{s:conclusions}

In this work, we have proposed {MISC, a combination technique method} to solve UQ
problems, optimizing both the deterministic and stochastic resolution levels simultaneously
to minimize the computational cost.
{A distinctive feature of MISC is that its construction is based on the notion
of profit of the mixed differences composing it, rather  than balancing the total error contributions arising from
 the deterministic and stochastic components.}
We have detailed a complexity analysis and derived
a convergence theorem showing that in certain cases the convergence of the method is essentially
dictated by the convergence properties of the deterministic solver.
We have then verified the effectiveness of the method proposed on
a couple of numerical test cases, comparing
its performance with other methods available in the literature. The results obtained
are encouraging, as they suggest that the proposed methodology is more effective than
the other methods considered here. The theoretical results have been also found to be
consistent with the numerical results to a satisfactory  extent.

As a final remark, we observe that the methodology presented here is not
limited to the spatial or temporal discretization parameters of the deterministic problem,
but could also be applied to other discretization parameters, such as smoothing parameters or
artificial viscosities.

\paragraph{Acknowledgement}
  F. Nobile and L. Tamellini received support from the
  Center for ADvanced MOdeling Science (CADMOS) and partial
  support by the Swiss National Science Foundation under the
  Project No. 140574 ``Efficient numerical methods for flow and
  transport phenomena in heterogeneous random porous media''.
  R. Tempone is a member of the KAUST Strategic Research Initiative,
  Center for Uncertainty Quantification in Computational Sciences and Engineering.

\appendix
\section{Proof of Theorem \ref{thm:misc_complexity}}
The following technical lemmas are needed in the convergence proof.
\begin{lemma}\label{lemma:bound-sum-with-int}
  For $\vec x \in (1,\infty)^D$, define $\lfloor \vec x \rfloor = \left(\lfloor x_i\rfloor\right)_{i=1}^D$.
  For any $f: (1,\infty)^D \to \rset$ and $g: (1,\infty)^D \to \rset_+$,
\[
    \sum_{\{\valpha \in \nset_+^D  \::\: f(\valpha) \leq 0\}} g(\valpha) =
    \int_{\{\vec x \in (1,\infty)^D  \::\: f(\lfloor \vec x \rfloor) \leq 0\}} g(\lfloor \vec x \rfloor) \di{\vec x}
\]
  holds. Moreover, if $g$ and $f$ are increasing, then
\[
    \sum_{\{\valpha \in \nset_+^D  \::\: f(\valpha) \leq 0\}} g(\valpha)
    \leq
    \int_{\{\vec x \in (1,\infty)^D  \::\: f(\vec x -1) \leq
      0\}} g(\vec x) \di{\vec x},
\]
  and if $g$ and $f$ are decreasing, then
  \[
    \sum_{\{\valpha \in \nset_+^D  \::\: f(\valpha) \leq 0\}} g(\valpha)
    \leq
    \int_{\{\vec x \in (1,\infty)^D \::\: f(\vec x) \leq
      0\}} g(\vec x-1) \di{\vec x}.
  \]
\end{lemma}
\begin{proof}
We have
  \begin{align*}
    \sum_{\{\valpha \in \nset_+^D  \::\: f(\valpha) \leq 0\}} g(\valpha)
   &= \sum_{\{\valpha \in \nset_+^D  \::\: f(\valpha) \leq 0\}}
    g(\valpha) \int_{\vec x \in [0,1]^{D}} \di{\vec x}\\
   &= \sum_{\{\valpha \in \nset_+^D  \::\: f(\valpha) \leq 0\}}
     \int_{\vec x \in [0,1]^{D}}
    g(\lfloor \valpha + \vec x \rfloor) \di{\vec x}\\
   &=
\int_{\{\vec x \in (1,\infty)^D  \::\: f(\lfloor \vec x \rfloor) \leq
      0\}} g(\lfloor \vec x \rfloor) \di{\vec x}.
  \end{align*}
Combining these inequalities with $x-1 \leq \lfloor x \rfloor \leq x$ finishes the proof.
\end{proof}

\begin{lemma}\label{lemma:bound-inf-sum-exp-lin}
Assume $\vec a \in \rset_+^D$, $\vec b \in \rset_+^D$ and $L>|\vec a|$. Then,
\[
\sum_{\{\vec x \in \nset_+^D \: :\: \sum_{i=1}^D a_i e^{b_i x_i} +
    b_i x_i > L \}} \exp\left(-\sum_{i=1}^D a_i
  e^{b_i x_i}\right) \leq
\left(\prod_{i=1}^D\frac{\exp(2a_i)}{a_i^2}\right)
\exp\left(-L\right) (L+1)^{2D+1}.
\]
\end{lemma}
\begin{proof}
Define the set
\[ \mathcal P = \left\{ \left(e^{b_i x_i}\right)_{i=1}^D \::\: \vec x
    \in \nset_+^D\right\},\]
and define $\lfloor \vec y \rfloor = \left(\lfloor y_i \rfloor\right)_{i=1}^D$. Then
\[
  \aligned \sum_{\{\vec x \in \nset_+^D \: :\: \sum_{i=1}^D a_i e^{b_ix} +
    b_i x_i > L \}} \exp\left(-\sum_{i=1}^D a_i
  e^{b_i x_i}\right)  &=
  \sum_{\{\vec y \in \mathcal P \: :\: \vec a \cdot \vec y +
    |\log(\vec y)| > L \}} \exp(-\vec a \cdot \vec y) \\
  &\leq \sum_{\{\vec y \in \mathcal P \: :\: \vec a \cdot (\lfloor
    \vec y \rfloor + \vec 1) +
    |\log(\lfloor \vec y\rfloor+\vec 1)| > L \}} \exp(-\vec a \lfloor \vec y \rfloor) \\
&\leq \sum_{\{\vec y \in \nset_+^D \: :\: \vec a \cdot  \vec y +
    |\log( \vec y+\vec 1)| > L-|\vec a| \}} \exp(-\vec a \cdot \vec y) \\
&\leq \int_{\{\vec y \in (1, \infty)^D \: :\: \vec a \cdot \vec y +
    |\log( \vec y+\vec 1)| > L-|\vec a| \}} \exp(-\vec a \cdot (\vec y - \vec
  1)) \di{\vec y}.
\endaligned
\]
Letting $z_i = a_i y_i + \log(y_i+1)$ and $p(z_i) = y_i \leq
\frac{z_i}{a_i}$, then
  \begin{align*}
  & \sum_{\{\vec x \in \nset_+^D \: :\: \sum_{i=1}^D a_i
    e^{b_ix} + b_i x_i > L \}} \exp{\left(-\sum_{i=1}^D a_i e^{b_i
        x_i}\right)} \\
  & = \exp(|\vec a|)\int_{\{\vec y \in (1, \infty)^D \: :\: \vec a
    \cdot \vec y + |\log(\vec y+\vec 1)| > L-|\vec a| \}} \exp(-\vec a
  \cdot \vec y - |\log(\vec y+ \vec 1)| + |\log(\vec y+\vec 1)|)
  \di{\vec y} \\
  &= \exp(|\vec a|) \int_{\{\vec z \in \otimes_{i=1}^D(a_i + \log(2),
    \infty) \: :\: |\vec z| > L-|\vec a| \}} \exp(- |\vec z|)
  \prod_{i=1}^D \left( \frac{p(z_i) + 1}{a_i +
      \frac{1}{p(z_i)+1}}\right)\di{\vec z} \\
  &\leq \exp(|\vec a|) \int_{\{\vec z \in \otimes_{i=1}^D(a_i +
    \log(2), \infty) \: :\: |\vec z| > L-|\vec a| \}} \exp(-|\vec z|)
  \prod_{i=1}^D \left( \frac{z_i +
      a_i}{a_i^2}\right)\di{\vec z} \\
  &\leq \left(\prod_{i=1}^D\frac{\exp(a_i)}{a_i^2}\right) \int_{\{\vec
    z \in \otimes_{i=1}^D(a_i + \log(2), \infty) \: :\: |\vec z| >
    L-|\vec a| \}}
  \exp(-|\vec z| + |\log(\vec z + \vec a)|)\di{\vec z} \\
  &= \left(\prod_{i=1}^D\frac{\exp(2 a_i)}{a_i^2}\right) \int_{\{\vec
    x \in \otimes_{i=1}^D(\log(2), \infty) \: :\: |\vec x| > L \}}
  \exp(-|\vec x| + |\log(\vec x)|)\di{\vec x} \\
  &\leq \left(\prod_{i=1}^D\frac{\exp(2 a_i)}{a_i^2}\right)
  \int_{\{\vec z \in \otimes_{i=1}^D(0, \infty) \: :\: |\vec z| > L
    \}} \exp(-|\vec z| + |\log(\vec z)|)\di{\vec z} .
  \end{align*}
Now let us prove, by induction on $D$, that
we have
\[
\int_{\{\vec z \in \rset_+^D \: :\:
  |\vec z| > L \}}
 \exp(-|\vec z| + |\log(\vec z)|)\di{\vec z} \leq \exp(-L)(L+1)^{2D-1}\,\,.
\]
For $D=1$, the inequality is a trivial equality that can be obtained
with integration by parts. Assume the inequality is true for $D$ and
let us prove it for $D+1$:
  \begin{align*}
\int_{\{\vec z \in \rset_+^{D+1} \: :\: |\vec z| > L
    \}} \exp(-|\vec z| + \log(\vec z ))\di{\vec z} &=
  \int_{L}^\infty y \exp(-y) \int_{\{\vec x \in
    \rset_+^{D} \}}
  \exp(-|\vec x| + \log(\vec x ))\di{\vec x} \di{y} \\
  &+\int_{0}^L y\exp(-y) \int_{\{\vec x \in
    \rset_+^{D} \::\: |\vec x| > L-y \}}
  \exp(-|\vec x| + \log(\vec x))\di{\vec x} \di{y} \\
  &\leq \exp(-L) (L+1) \\
  &+ \int_{0}^L y\exp(-y) \exp(-L+y)(L-y+1)^{2D-1} \di{y}
  \\
  &\leq \exp(-L)\left( L+1+ L^2 (L+1)^{2D-1}\right)\\
&\leq \exp(-L)\left( L+1\right)\left(1+ L (L+1)^{2D-1}\right)\\
&\leq \exp(-L) (L+1)^{2(D+1)-1}\,\,.
  \end{align*}
Finally, substituting back, we get the result.
\end{proof}

\begin{definition}
Given $\vec a \in \rset_+^D$ and $A>0$, let $\mathfrak{n}(\vec a, A)$ denote
the number of occurrences of $A$ in $\vec a$,
\[
\mathfrak{n}(\vec a, A) = \#\{i = 1, \ldots, d \::\: a_i = A\}.
\]
\end{definition}

\begin{lemma}
Assume $k \in \nset$, $\vec a \in \rset_+^D, L>|\vec a|$.
Then, the following bounds hold true:
\begin{align*}
  \int_{\{\vec x \in \rset_+^D \::\: |\vec x| > L\}} \exp(-\vec a \cdot\vec x)\di{\vec x}
  	\leq \mathfrak{B}_D(\vec a) \exp(-\min(\vec a) L) L^{\mathfrak{n}(\vec a, \min(\vec a))-1}
\end{align*}
where $\mathfrak{B}_D(\vec a)$ is a positive constant independent of $L$.
\end{lemma}
\begin{proof}
  See \cite[Lemma B.3]{abdullatif.etal:MultiIndexMC} for a proof of
  the inequality and the value of $\mathfrak{B}_D(\vec a)$. Moreover,
  a proof of a consistent equality for the case $\vec a=\vec 1$ can be
  found in \cite[Proposition 2.3]{griebel:tensor}.
\end{proof}

\begin{lemma}
Assume $k \in \nset$, $\vec a \in \rset_+^D, L>|\vec a|$.
Then, the following bound holds:
\begin{align*}
  \int_{\{\vec x \in \rset_+^D \: :\: |\vec x| \leq L \}}
  \exp\left(\vec a \cdot \vec x \right) \left(L - |\vec x| \right)^k \di{\vec x}
  \leq \mathfrak A_{D}(\vec a, k) \exp(\max(\vec a) L) L^{\mathfrak{n}(\vec a, \max(\vec a))-1},
\end{align*}
where
\begin{align*}
& \mathfrak A_{D}(\vec a, k) = \frac{k!}{(\mathfrak{n}(\vec a, \max(\vec a))-1)! \max(\vec a)^{k+1}}
  \left(\prod_{a_i < \max(\vec a)} \frac{1}{ \max(\vec a) -
  a_i}\right).
\end{align*}
\end{lemma}
\begin{proof}
  Without loss of generality, assume that
  $a_{i} \geq a_{i+1}$ for all $i=1 \ldots D$, such that
  $a_1 = \max(\vec a)$.  We prove the result by induction on $D$. For
  $D=1$, we have
\[
  \aligned \int_0^L \exp\left(a x \right) \left(L - x \right)^k \di{x}
  &= \frac{k!}{a^{k+1}}\left(\exp(aL) - \sum_{i=0}^{k} \frac{a^i L^{i}}{i!}\right) \\
&\leq \frac{k! \exp(aL)}{a^{k+1}}.
\endaligned
\]

Next, assume that the result is valid for a given $D>1$ and $\vec
a\in\rset_+^D$ where $a_{i} \geq a_{i+1}$ for all
$i=1 \ldots D$, such that $a_1 = \max(\vec a)$. Let $b \leq a_1$ and
define a new vector $\widetilde{\vec a} = (\vec a, b) \in
\rset_+^{D+1}$. We have
  \begin{align*}
& \int_{\{(\vec x, y) \in \rset_+^{D+1} \: :\: y +
    |\vec x| \leq L \}} \exp\left(b y + \vec a \cdot \vec x
  \right)
  \left(L - y - |\vec x| \right)^k \di{y} \di{\vec x} \\
  &=\int_0^L {\exp\left(b y \right)} \int_{\{\vec x \in
    \rset_+^{D} \: :\: |\vec x| \leq L-y \}} \exp\left(\vec a \cdot
    \vec x \right) \left(L - y -|\vec x| \right)^k \di{\vec x}
  \di{y} \\
  &\leq \mathfrak{A}_D(\vec a, k) {\exp(a_1 L)} \int_0^L {\exp\left((b-a_1) y \right)} (L-y)^{\mathfrak{n}(\vec a, a_1)-1}
  \di{y}.
  \end{align*}
We distinguish between two cases:
\begin{enumerate}
\item $b < a_1$ then
 $\mathfrak{n}(\widetilde{\vec a}, a_1) = \mathfrak{n}(\vec a, a_1)$ and
\[
  \aligned \int_0^L {\exp\left(-(a_1-b)
      y \right)} (L-y)^{\mathfrak{n}(\vec a, a_1)-1} \di{y}
  &\leq L^{\mathfrak{n}(\vec a, a_1)-1}
  \int_0^\infty {\exp\left(-(a_1-b) y \right)}  \di{y} \\
&\leq L^{\mathfrak{n}(\widetilde{\vec a}, a_1)-1} \frac{1}{a_1-b},
\endaligned\]
and in this case
\[\aligned \mathfrak{A}_D(\vec a, k) \frac{1}{a_1-b} &=
\frac{k!}{(\mathfrak{n}(\vec a, a_1)-1)! a_1^{k+1}}
  \left(\prod_{a_i < a_1} \frac{1}{ a_1 - a_i}\right) \left(\frac{1}{a_1-b}\right)\\
&= \mathfrak{A}_{D+1}(\widetilde{\vec a}, k).
\endaligned\]
\item $b = a_1$ then
  $\mathfrak{n}(\widetilde{\vec a}, a_1) = \mathfrak{n}(\vec a,
  a_1)+1$ and
\[
\aligned
 \int_0^L (L-y)^{\mathfrak{n}(\vec a, a_1)-1}
  \di{y} &= \frac{L^{\mathfrak{n}(\vec a, a_1)}}{\mathfrak{n}(\vec a, a_1)}
= \frac{L^{\mathfrak{n}(\widetilde{\vec a}, a_1)-1}}{\mathfrak{n}(\widetilde{\vec a}, a_1)-1},
\endaligned
\]
and again
\[\aligned \frac{1}{\mathfrak{n}(\widetilde{\vec a}, a_1)-1}  \mathfrak{A}_D(\vec a, k)  &=
\frac{1}{\mathfrak{n}(\widetilde{\vec a}, a_1)-1}  \cdot \frac{k!}{(\mathfrak{n}(\vec a, a_1)-1)! a_1^{k+1}}
  \left(\prod_{a_i < a_1} \frac{1}{ a_1 - a_i}\right)\\
&= \mathfrak{A}_{D+1}(\widetilde{\vec a}, k).
\endaligned\]
\end{enumerate}

\end{proof}

\misccomplexity*
\begin{proof}

The bounds \eqref{eq:deltaW_ass} and \eqref{eq:deltaE_ass} can be obtained by elementary algebraic operations
combining Assumptions \ref{assump:growth_of_dof} and \ref{assump:dW_dE_model}; for instance,
\begin{align*}
\Delta W_\valpha^\textnormal{det} &
\leq C_{\textnormal{work}}^{\textnormal{det}} \prod_{i=1}^D h_i^{-\widetilde{\gamma}_i}
= C_{\textnormal{work}}^{\textnormal{det}} \prod_{i=1}^D (h_0 2^{-\alpha_i})^{-\widetilde{\gamma}_i}
= C_{\textnormal{work}}^{\textnormal{det}} h_0^{-|\widetilde{\vec \gamma}|} \prod_{i=1}^D 2^{\widetilde{\gamma}_i \alpha_i}
= C_{\textnormal{work}}^{\textnormal{det}} h_0^{-|\widetilde{\vec \gamma}|} \prod_{i=1}^D e^{\widetilde{\gamma}_i \alpha_i \log 2}, \\
\Delta W_\vbeta^\textnormal{stoc}  &
\leq C_{\textnormal{work}}^{\textnormal{stoc}} \prod_{n=1}^N 2^{\beta_n}
= C_{\textnormal{work}}^{\textnormal{stoc}} \prod_{n=1}^N e^{\beta_n \log 2},
\end{align*}
from which \eqref{eq:deltaW_ass} follows by setting
$\CWork = C_{\textnormal{work}}^{\textnormal{det}}
h_0^{-|\widetilde{\vec \gamma}|}
C_{\textnormal{work}}^{\textnormal{stoc}}$.
The proof is then divided into two steps.

\paragraph{Step 1: Work Estimate}
Observe that $\Xi_i = \frac{\gamma_{i}}{\gamma_i+\dRate{i}}  < 1$ for
all $i=1,\ldots D$ and that
$\mathfrak{z} = \mathfrak{n}(\vXi, \chi)$.  Thanks to
equations \eqref{eq:work_decomp} and \eqref{eq:deltaW_ass}, and using Lemma \ref{lemma:bound-sum-with-int},
the total work satisfies
\begin{align*}
  &\work{\mathcal{I}^*(L)} =
\sum_{(\valpha, \vbeta) \in \mathcal{I}^*(L)} \Delta
  W_{\valpha, \vbeta} \\
  \leq & {\CWork}
         \sum_{\left\{
         (\valpha, \vbeta) \in \nset_+^{D+N} \::\:
         \sum_{i=1}^D (\dRate{i} + \gamma_i)\alpha_i
         + \sum_{j=1}^N \delta \beta_j + g_j e^{\delta \beta_j} \leq
          L
         \right\}}
         \exp\left(\sum_{i=1}^D \gamma_i \alpha_i + \delta |\vbeta|
         \right) \\
  \leq & {\CWork}
         \int_{\left\{(\valpha, \vbeta) \in (1,\infty)^{D+N} \::\:
         \sum_{i=1}^D (\dRate{i} + \gamma_i)(\alpha_i-1)
         + \sum_{j=1}^N \delta (\beta_j-1) + g_j e^{\delta (\beta_j-1)} \leq
          L
         \right\}}
         \exp\left(\sum_{i=1}^D \gamma_i \alpha_i + \delta |\vbeta|
         \right) \di{\valpha}\di{\vbeta}.
\end{align*}
Next, let $\overline{\beta_j} = g_j e^{\delta(\beta_j-1)}$ and $\overline{\alpha_i} = (\dRate{i} + \gamma_i) (\alpha_i-1)$.
We have
\begin{align*}
  \work{\mathcal{I}^*(L)} \leq
   & {\CWork} \left(\prod_{j=1}^N \frac{2}{g_j \delta}\right)
      \left(\prod_{i=1}^D \frac{\exp(\gamma_i)}{\dRate{i} + \gamma_i}
      \right) \\
      &\qquad \int_{\left\{ (\overline{\valpha}, \overline{\vbeta}) \in
          \rset_+^D \times (\otimes_{j=1}^N (g_j, \infty))
          \::\:
          |\overline{\valpha}|
          + |\overline{\vbeta}| + |\log \overline{\vbeta}| \leq
          L + |\log \vec g| \right\}}
      \exp\left(\vXi \cdot \overline \valpha \right) \di{\overline{\valpha}}\di{\overline{\vbeta}}.
    \end{align*}
Dropping the over-line notation and defining $\widetilde L = L + |\log
\vec g|$ and $\mathscr{C}_{\text{W},1}$ to be the constant factor, we obtain
\begin{align*}
  &\work{\mathcal{I}^*(L)}
  \leq \mathscr{C}_{\text{W},1}
    \int_{\left\{ ({\valpha}, {\vbeta}) \in
\rset_+^D \times (\otimes_{j=1}^N (g_j, \infty)) \::\:
    |{\valpha}|
    + |{\vbeta}| + |\log {\vbeta}| \leq  \widetilde L \right\}}
    \exp\left(\vXi \cdot  \valpha  \right) \di{{\valpha}}\di{{\vbeta}}\\
  &= \mathscr{C}_{\text{W},1}
    \int_{\left\{ {\vbeta} \in \otimes_{j=1}^N (g_j, \infty) \::\:
    |{\vbeta}| + |\log{\vbeta}| \leq  \widetilde L \right\}}  \int_{\left\{
    {\valpha} \in \rset_+^D \::\:
    |\valpha| \leq  \widetilde L - |{\vbeta}| - |\log{\vbeta}| \right\}}
    \exp\left(\vXi \cdot  \valpha  \right) \di{\valpha} \di{\vbeta} \\
  &\leq \mathscr{C}_{\text{W},1}
    \mathfrak{A}_D\left( \vXi, 0\right)
    \int_{\left\{  {\vbeta} \in \otimes_{j=1}^N (g_j, \infty) \::\:
    |{\vbeta}| + |\log{\vbeta}| \leq  \widetilde L \right\}} \exp\Big( \chi \left(\widetilde L - |{\vbeta}| -
    |\log{\vbeta}|\right)\Big) \left(\widetilde L - |{\vbeta}| -
    |\log{\vbeta}|\right)^{\mathfrak{z}-1} \di{\vbeta}
\end{align*}
Define $\mathscr{C}_{\text{W},2} = \mathscr{C}_{\text{W},1}
    \mathfrak{A}_D\left( \vXi, 0\right) \exp(\chi \widetilde L)$, then
\begin{align*}
\work{\mathcal{I}^*(L)}
  &\leq \mathscr{C}_{\text{W},2}
    \int_{\left\{  {\vbeta} \in \otimes_{j=1}^N (g_j, \infty) \::\:
    |{\vbeta}| + |\log{\vbeta}| \leq  \widetilde L \right\}} \exp\left(- \chi \left(|{\vbeta}| +
    |\log{\vbeta}|\right)\right) \left(\widetilde L - |{\vbeta}| -
    |\log{\vbeta}|\right)^{\mathfrak{z}-1} \di{\vbeta}\\
  &\leq \mathscr{C}_{\text{W},2}
     \left({\widetilde L} - |\vec g| -
    |\log \vec g |\right)^{\mathfrak{z}-1}
    \int_{\left\{  {\vbeta} \in \otimes_{j=1}^N (g_j, \infty) \::\:
    |{\vbeta}| + |\log{\vbeta}| \leq  \widetilde L \right\}} \exp\left(-\chi \left( |{\vbeta}| +
    |\log{\vbeta}|\right)\right)  \di{\vbeta}.
\end{align*}
Since $\chi > 0$,  the previous integral is bounded  for all
$\widetilde L$ and we have
\begin{align*}
    \work{\mathcal{I}^*(L)} &\leq \mathscr{C}_{\text{W}} \exp(\chi
                            L) \left( L - |\vec g|\right)^{\mathfrak{z}-1}\leq \mathscr{C}_{\text{W}} \exp(\chi
                            L) L^{\mathfrak{z}-1},
\end{align*}
where
\[ \mathscr{C}_{\text{W}} = {\CWork} \left(\prod_{j=1}^N \frac{2
    g_i^{\chi}}{g_j \log 2}\right)
      \left(\prod_{i=1}^D \frac{\exp(\gamma_i)}{\dRate{i} + \gamma_i} \right)
\mathfrak{A}_D\left( \vXi, 0\right)
\int_{\left\{{\vbeta} \in \otimes_{j=1}^N (g_j, \infty) \right\}} \exp\left(-\chi \left( |{\vbeta}| +
  |\log{\vbeta}|\right)\right)  \di{\vbeta}.
\]
Substituting \eqref{eq:opt-L} yields
  \begin{align*}
  \work{\mathcal{I}^*(L)}
    &\leq W_{\max} \left(1 - \frac{(\mathfrak{z}-1)
\log \left( \frac{\log\left(\frac{W_{\max}}{\mathscr{C}_{\text{W}}}\right)}{\chi}
      \right)
      }
{\log\left(\frac{W_{\max}} {\mathscr{C}_{\text{W}}}\right)}
\right)^{\mathfrak{z}-1}.
  \end{align*}
From here it is easy to see that if \eqref{eq:maxW_ass} is
satisfied, then \eqref{eq:misc_complexity_work} follows.

\paragraph{Step 2: Error Estimate}
Thanks to equations \eqref{eq:error-decomp} and \eqref{eq:deltaE_ass},
the total error satisfies
  \begin{align*}
    &\error{\mathcal{I}^*(L)}
    \leq \sum_{(\valpha,\vbeta) \notin \mathcal I^*} \Delta E_{\valpha, \vbeta}\\
    \leq & \CError \sum_{\left\{ (\valpha, \vbeta) \in \nset_+^{D+N} \::\:
           \sum_{i=1}^D (\dRate{i} + \gamma_i)\alpha_i
           + \sum_{j=1}^N \delta \beta_j + g_j e^{\delta{\beta_j}} > L \right\}}
           \exp\left(
           -\sum_{i=1}^D \dRate{i} \alpha_i - \sum_{j=1}^N g_j
           e^{\delta \beta_j} \right) \\
    = & \CError \sum_{\left\{ (\valpha,\vbeta) \in \nset_+^{D+N}
        \::\:
        \sum_{i=1}^D (\dRate{i} + \gamma_i)\alpha_i > L \right\}}
        \exp\left(
        -\sum_{i=1}^D \dRate{i} \alpha_i - \sum_{j=1}^N g_j
        e^{\delta \beta_j} \right) \\
         &+ \CError \sum_{\left\{  \valpha \in \nset_+^{D} \::\:
           \sum_{i=1}^D (\dRate{i} + \gamma_i)\alpha_i  \leq L
           \right\}}
           \sum_{\left\{ \vbeta \in \nset_+^{N} \::\:
           \sum_{j=1}^N \delta \beta_j + g_j e^{\delta{\beta_j}} > L
           - \sum_{i=1}^D (\dRate{i} + \gamma_i)\alpha_i
           \right\}}
           \exp\left(-\sum_{i=1}^D \dRate{i} \alpha_i -\sum_{j=1}^N g_j e^{\delta \beta_j} \right).
   \end{align*}
Looking at the first term, let
 $\eta_i = \frac{\dRate{i}}{\gamma_i+\dRate{i}} < 1$ and $\veta =
\left(\eta_i\right)_{i=1}^D$ and note that
 $\mathfrak{z} = \#\left\{ i = 1 \ldots D:  \eta_i =
   \min(\veta)\right\}$. Then
\begin{align*}
  &\sum_{\left\{ (\valpha,\vbeta) \in \nset_+^{D+N}
    \::\:
    \sum_{i=1}^D (\dRate{i} + \gamma_i)\alpha_i > L \right\}}
    \exp\left(
    -\sum_{i=1}^D \dRate{i} \alpha_i - \sum_{j=1}^N g_j
    e^{\delta \beta_j} \right)
  \\
  &= \left(\sum_{\vbeta \in \nset_+^{N}}
    \exp\left( - \sum_{j=1}^N g_j e^{\delta \beta_j} \right)\right)
    \left(\sum_{
    \left\{ \valpha \in \nset_+^{D}
    \::\:
    \sum_{i=1}^D (\dRate{i} + \gamma_i)\alpha_i > L \right\}}
    \exp\left(-\sum_{i=1}^D \dRate{i} \alpha_i \right)\right)\\
  &\leq \mathscr{C}_{\text{E},1}
\int_{
    \left\{ \valpha \in (1,\infty)^{D}\::\:
    \sum_{i=1}^D (\dRate{i} + \gamma_i)\alpha_i > L \right\}}
    \exp\left(-\sum_{i=1}^D \dRate{i} \left(\alpha_i-1\right) \right) \di{\valpha}\\
  &= \mathscr{C}_{\text{E},1} \left(\prod_{i=1}^D
\frac{\exp(\dRate{i})}{\dRate{i} + \gamma_i} \right)
\int_{
    \left\{ \vec x \in \otimes_{i=1}^D (\dRate{i}+\gamma_i,\infty) \::\:
    |\vec x| > L \right\}}
    \exp\left(-\sum_{i=1}^D \frac{\dRate{i}}{\dRate{i}+\gamma_i} x_i \right)
    \di{\vec x} \\
&\leq \mathscr{C}_{\text{E},2} \exp\left(-\min(\veta) L\right) L^{\mathfrak{z}-1},
\end{align*}
where
\[
  \mathscr{C}_{\text{E},2} =  \mathfrak{B}_D(\veta)
  \left(\prod_{i=1}^D \frac{\exp(\dRate{i})}{\dRate{i} + \gamma_i}
  \right) \sum_{\vbeta \in \nset_+^{N}}
    \exp\left( - \sum_{j=1}^N g_j e^{\delta \beta_j} \right).
\]
For the second term, letting
$H = L - \sum_{i=1}^D (\dRate{i}+\gamma_i)\alpha_i$,
we can bound the sum using Lemma~\ref{lemma:bound-inf-sum-exp-lin}:
\[
\sum_{\left\{ \vbeta \in \nset_+^{N} \::\:
           \sum_{j=1}^N \delta \beta_j + g_j e^{\delta{\beta_j}} > H \right\}}
           \exp\left(-\sum_{j=1}^N g_j e^{\delta \beta_j} \right)
\leq \left(\prod_{j=1}^N \frac{\exp(2g_j)}{g_j^2} \right) \exp(-H)
(H+1)^{2N-1}.
\]
Defining $\mathscr C_{E,3} = \prod_{j=1}^N {\exp(2g_j)}{g_j^{-2}}$ and
substituting back
\begin{align*}
 &\sum_{\left\{ \valpha \in \nset_+^{D} \::\: \sum_{i=1}^D
      (\dRate{i} + \gamma_i)\alpha_i \leq L \right\}}
  \exp\left(-\sum_{i=1}^D \dRate{i} \alpha_i \right)
  \sum_{\left\{ \vbeta \in \nset_+^{N} \::\: \sum_{j=1}^N \delta
      \beta_j + g_j e^{\delta{\beta_j}} > L - \sum_{i=1}^D (\dRate{i}
      + \gamma_i)\alpha_i \right\}}
  \exp\left(-\sum_{j=1}^N g_j e^{\delta \beta_j} \right) \\
  &\leq \mathscr C_{E,3}%
\sum_{\left\{ \valpha \in \nset_+^{D} \::\: \sum_{i=1}^D
      (\dRate{i} + \gamma_i)\alpha_i \leq L \right\}}
  \exp\left(-L + \sum_{i=1}^D \gamma_i \alpha_i \right)
  \left(L +1 - \sum_{i=1}^D (\dRate{i}+\gamma_i)\alpha_i\right)^{2N-1}\\
  &= \mathscr C_{E,3}%
  \int_{\left\{ \valpha \in (1,\infty)^{D} \::\: \sum_{i=1}^D
      (\dRate{i} + \gamma_i)\lfloor \alpha_i \rfloor \leq L \right\}}
  \exp\left(-L + \sum_{i=1}^D \gamma_i \lfloor \alpha_i \rfloor \right)
  \left(L +1- \sum_{i=1}^D (\dRate{i}+\gamma_i)\lfloor \alpha_i
    \rfloor\right)^{2N-1}\di{\valpha}\\
&\leq \mathscr C_{E,3}%
  \int_{\left\{ \valpha \in (1,\infty)^{D} \::\: \sum_{i=1}^D
      (\dRate{i} + \gamma_i)(\alpha_i-1) \leq L \right\}}
  \exp\left(-L + \sum_{i=1}^D \gamma_i  \alpha_i \right)
\left(L+1 - \sum_{i=1}^D (\dRate{i}+\gamma_i)( \alpha_i-1)\right)^{2N-1}
\di{\valpha}\\
&= \mathscr C_{E,3}%
\left(\prod_{i=1}^D \frac{\exp(\gamma_i)}{\gamma_i + \dRate{i}} \right)
  {\exp ( - L )}
  \int_{\left\{ \valpha \in \rset_+^{D} \::\: |\vec x| \leq L \right\}}
  \exp\left(\vXi \cdot \vec x \right)
\left(L+1 - |\vec x|\right)^{2N-1}
\di{\vec x}\\
  &\leq \mathscr{C}_{\text{E}, 4} {\exp ( (\chi-1) L )}
  L^{\mathfrak{z}-1},
\end{align*}
where
\[\mathscr{C}_{\text{E}, 4} = \left(\prod_{j=1}^N
    \frac{\exp(2g_j)}{g_j^2} \right) \left(\prod_{i=1}^D \frac{\exp(\gamma_i)}{\gamma_i + \dRate{i}} \right) \mathfrak{A}_D(\vXi, 2N-1).\]
Finally, noting that
\[ \chi-1 = -\min(\veta), \]
we have the error estimate
\begin{align*}%
  \error{\mathcal{I}^*(L)} \leq
  \CError \left(\mathscr{C}_{\text{E}, 2}
  +\mathscr{C}_{\text{E}, 4} \right) \exp(-\min(\veta) L) L^{\mathfrak{z}-1}.
\end{align*}
Then, substituting $L$ from \eqref{eq:opt-L} and evaluating the limit
gives~\eqref{eq:misc_complexity_error}.
\end{proof}

\bibliographystyle{elsarticle-num}
 \end{document}